\documentclass[aap]{imsart}

\RequirePackage{amsthm,amsmath,amsfonts,amssymb}
\RequirePackage[numbers]{natbib}
\RequirePackage[colorlinks,citecolor=blue,urlcolor=blue]{hyperref}%% uncomment this for coloring bibliography citations and linked URLs
\RequirePackage{graphicx}%% uncomment this for including figures
\RequirePackage{bbm}
\RequirePackage{subcaption}

\startlocaldefs
\newcommand{\DD}{\mathbb{D}}
\newcommand{\R}{\mathbb{R}}

\newcommand{\N}{\mathbb{N}}
\newcommand{\EE}{\mathbb{E}}

\newcommand{\PP}{\mathbb{P}}
\newcommand{\F}{\mathcal{F}}

\newcommand{\Pcal}{\mathcal{P}}

\newcommand{\Lcal}{\mathcal{L}}
\newcommand{\Fcal}{\mathcal{F}}
\newcommand{\Hcal}{\mathcal{H}}

\newcommand{\1}{\mathbbm{1}}

\newcommand{\Mcal}{\mathcal{M}}
\newcommand{\Kcal}{\mathcal{K}}

\newcommand{\Ncal}{\mathcal{N}}

\newcommand{\dt}{\partial_t}
\newcommand{\dx}{\partial_x}

\newcommand{\dz}{\partial_z}
\newcommand{\dzz}{\partial_{zz}}
\newcommand{\norm}[1]{\left\lVert#1\right\rVert}
\newcommand{\s}{\sigma}

\theoremstyle{plain}
\newtheorem{theorem}{Theorem}[section]
\newtheorem{proposition}[theorem]{Proposition}
\newtheorem{definition}[theorem]{Definition}
\newtheorem{conjecture}[theorem]{Conjecture}
\newtheorem{lemma}[theorem]{Lemma}
\endlocaldefs

\begin{document}

\begin{frontmatter}
%%%%%%%%%%%%%%%%%%%%%%%%%%%%%%%%%%%%%%%%%%%%%%
%%                                          %%
%% Enter the title of your article here     %%
%%                                          %%
%%%%%%%%%%%%%%%%%%%%%%%%%%%%%%%%%%%%%%%%%%%%%%
\title{Convergence and Wave Propagation for a System of Branching Rank-Based Interacting Brownian Particles}
%\title{A sample article title with some additional note\thanksref{T1}}
\runtitle{A System of Branching Rank-Based Interacting Brownian Particles}
%\thankstext{T1}{A sample of additional note to the title.}

\begin{aug}
%%%%%%%%%%%%%%%%%%%%%%%%%%%%%%%%%%%%%%%%%%%%%%%
%% Only one address is permitted per author. %%
%% Only division, organization and e-mail is %%
%% included in the address.                  %%
%% Additional information can be included in %%
%% the Acknowledgments section if necessary. %%
%% ORCID can be inserted by command:         %%
%% \orcid{0000-0000-0000-0000}               %%
%%%%%%%%%%%%%%%%%%%%%%%%%%%%%%%%%%%%%%%%%%%%%%%
\author[A]{\fnms{Mete} ~\snm{Demircigil}\ead[label=e1]{mete.demircigil@math.cnrs.fr}\orcid{0009-0002-0311-5578}},
\author[B]{\fnms{Milica}~\snm{Tomasevic}\ead[label=e2]{milica.tomasevic@polytechnique.edu}}
%%%%%%%%%%%%%%%%%%%%%%%%%%%%%%%%%%%%%%%%%%%%%%
%% Addresses                                %%
%%%%%%%%%%%%%%%%%%%%%%%%%%%%%%%%%%%%%%%%%%%%%%
\address[A]{Department of Mathematics, University of Arizona\printead[presep={,\ }]{e1}}

\address[B]{CMAP, CNRS, \'{E}cole polytechnique, Institut Polytechnique de Paris\printead[presep={,\ }]{e2}}
\end{aug}

\begin{abstract}
    In this work we study a branching particle system of diffusion processes on the real line interacting through their rank in the system. Namely, each particle follows an independent Brownian motion, but only $K\geq 1$ particles on the far right are allowed to branch with constant rate, whilst the remaining particles have an additional positive drift of intensity $\chi>0$. This is the so called \textit{Go or Grow} hypothesis, which serves as an elementary hypothesis to model cells in a capillary tube moving upwards a chemical gradient. \\
   Despite the discontinuous character of the coefficients for the movement of particles and their demographic events, we first obtain the limit behavior of the  population as $K \to \infty$ by weighting the individuals by $1/K$. \\
   Then, on the microscopic level when $K$ is fixed, we investigate numerically the speed of propagation of the particles and recover a threshold behavior according to the parameter $\chi$ consistent with the already known behavior of the limit. Finally, by studying numerically the ancestral lineages we categorize the traveling fronts as \textit{pushed} or \textit{pulled} according to the critical parameter $\chi$.
\end{abstract}



\end{frontmatter}

%%%%%%%%%%%%%%%%%%%%%%%%%%%%%%%%%%%%%%%%%%%%%%
%% Please use \tableofcontents for articles %%
%% with 50 pages and more                   %%
%%%%%%%%%%%%%%%%%%%%%%%%%%%%%%%%%%%%%%%%%%%%%%
\tableofcontents

%%%%%%%%%%%%%%%%%%%%%%%%%%%%%%%%%%%%%%%%%%%%%%
%%%% Main text entry area:

\section{Introduction}

In this work we study a rank-based interacting and branching  stochastic particle system in $\R$.

\begin{figure}[t]
\begin{center}
\includegraphics[width=\linewidth]{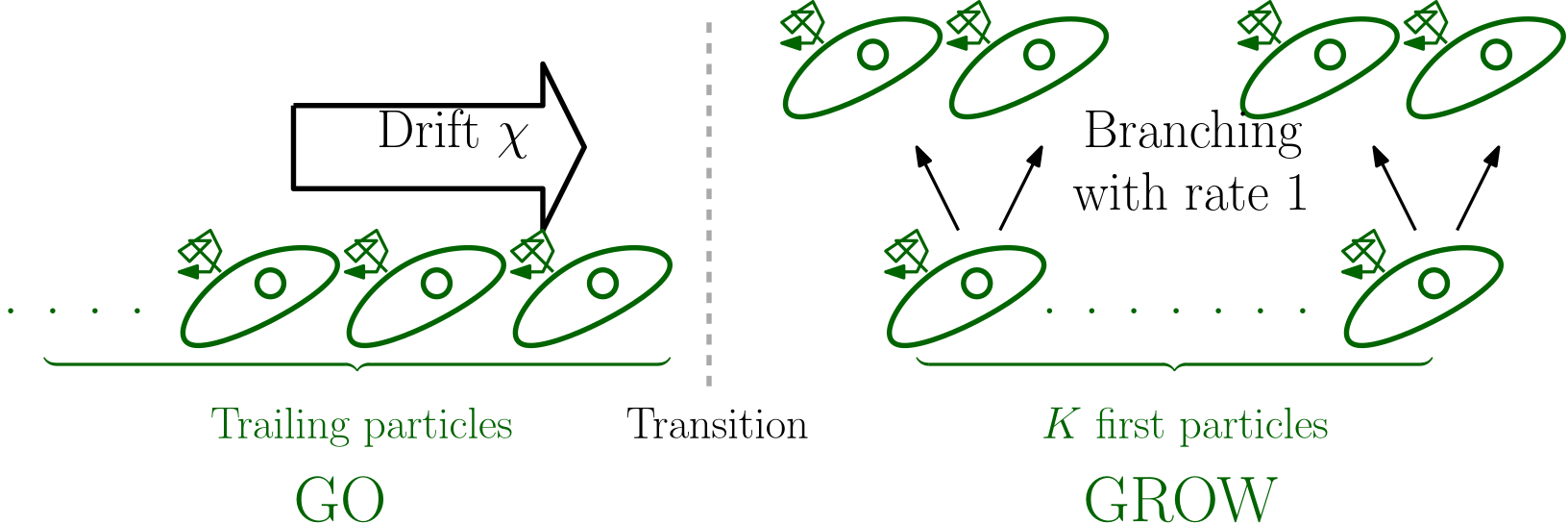}
\caption{Cartoon representation of the model. The first $K$ cells  are in the \textit{Grow} regime and branch with rate 1, whilst the trailing cells are in the \textit{Go} regime and drift with a drift coefficient $\chi$. Moreover, all cells undergo Brownian motion. }
\label{fig:cartoon}
\end{center}
\end{figure}

Our model is motivated by exhibiting traveling wave-like behavior for a population of cells in a capillary tube that move from one end of the tube by following gradients of a chemical nutrient \cite{adler1966,saragosti2011}: A population of the bacteria \textit{Escherichia coli} placed at the left end of the tube starts consuming and depleting a nutrient (\textit{e.g.} oxygen or galactose), which triggers a collective rightward movement under the form of a traveling wave towards areas of higher concentration. This phenomenon of directed motion is known as \textit{chemotaxis}. Moreover, the study \cite{cochet2021} has shown that in a similar, but radial setting involving  the amoeba \textit{Dictyostelium discoideum}, cell division, in combination with chemotaxis, plays a crucial role in the emergence of a traveling wave.

The model that we propose in this article can be seen as an elementary stochastic particle system model combining biased Brownian motion (as a model of chemotaxis) and particle branching (as a model of cell division). Namely, let us fix $K\geq 1$ and start from $N\geq K$ particles in some distinct positions in $\R$. Let us define the rank of a particle as its position in the queue  starting from the rightmost particle, \textit{i.e.} for $x_1, \dots, x_N \in \R$, 
$$rank(x_i) = \sum_{j=1}^N \1\{x_j\geq x_i\}.$$
All the particles move according to independent Brownian motions and the ones, whose rank is larger than $K$, have a positive constant drift $\chi>0$, whilst the particles, whose rank is smaller or equal to $K$, can branch with constant rate $1$. When a particle branches, a new one is added at its position moving according to a new independent Brownian motion. As the Brownian motions get instantaneously separated, there is no confusion with the rank of the new particle after its birth. These two different mechanisms will be called throughout this paper \textit{Go} or \textit{Grow}, which is a modeling hypothesis proposed in \cite{cochet2021,demircigil2022} for a parabolic PDE (although the initial occurrence \cite{hatzikirou2012} of '\textit{Go} or \textit{Grow}' describes a slightly different mechanism in glioma cells). \textit{Go} means that, whilst the particles, whose rank is high, \textit{i.e.} queuing the population, do not divide, their movement is polarized to the right via their positive drift $\chi>0$ and thus they contribute to a displacement of the population to the right. \textit{Grow} means that the particles with low rank, \textit{i.e.} who are heading the population, do not have a positive drift, yet, they are able to branch with rate 1 and as such spread the population to the right by creating mass at the front.

Given the biological context \cite{cochet2021,demircigil2022},  the cells that are subject to the \textit{Grow} regime are in an accommodating environment with a sufficient concentration of nutrient, sustaining their normal metabolism. Thus, they are able to undergo cell division and do not have an incentive to move towards zones of higher nutrient concentration. On the other hand, cells subject to the \textit{Go} regime find themselves in a zone with less nutrient and as such start a directed motion towards zones of higher nutrient concentration. 

The parameter $K$, which is fixed on the microscopic level, regulates the switch between the two regimes. To extend the biological analogy, we can consider that the environment is initially in an accommodating state with a sufficient concentration of nutrient. Yet, as the $K$ first cells move through that environment, the concentration of nutrient is depleted, leaving the trailing cells in a low concentration environment.

Our model lies at the intersection of two different well-studied classes of stochastic processes. The first class are so-called rank-based diffusion processes: A system of finite and constant number of particles interact with each other through their rank in the system, which determines their individual behavior (drift and diffusion), see \textit{e.g.}  \cite{ichiba2013} and its Introduction for references. As an example, we mention the Atlas model proposed in \cite{fernholz2002}: $N$ particles diffuse and the last one "pushes" the whole population, by having a positive drift. Later on, generalizations of the Atlas model have been proposed (see for instance \cite{banner2005,pal2008}). These systems also arise in the context of approximation of some  nonlinear evolution problems by interacting particle systems, see \textit{e.g.}  \cite{jourdain2013, jourdain2000, jourdain2000bis, BossyTalay}.\\
The other class are the rank-dependent branching Brownian motions following the terminology \cite{groisman2021}: There, the particles undergo Brownian motion and depending on their rank, they can divide (and possibly also die). Among this class, we mention the $N$-Branching Brownian Motion ($N$-BBM) introduced in \cite{maillard2016} (see also \cite{demasi2019,berestycki2024}): $N$ particles undergo Brownian motion and each particle divides with rate equal to one. When a particle divides the left-most particle is removed from the system.

Our first goal is to find the behavior of the population in the limit as $K\to \infty$,  weighting the individuals by $1/K$.  The main difficulty is the discontinuous character of our coefficients for the movement of particles and their demographic events. That is why, inspired by the works of Jourdain and his collaborators, see \textit{e.g.}  \cite{jourdain2000,jourdain2013}, we need to work with the empirical cumulative distribution function, which enables a  discrete integration by parts argument.
An additional difficulty here is the uniqueness of the limit that we tackle in the following way: 
First, we start by proving that the cumulative distribution function satisfies a mild form, from which we can deduce a regularity result. This regularity enables us to work then with the density function. Uniqueness is finally established via a contraction argument on the position of the threshold switching between the \textit{Go} and the \textit{Grow} regime.

Our other goal is to show via numerical investigation the relevance of the present model to study traveling wave-like behavior in cells undergoing chemotaxis in a capillary tube \cite{adler1966,saragosti2011,cochet2021}. On the one hand, we give numerical evidence for the emergence of traveling waves in our model. On the other hand, using a newly proposed methodology proposed independently in \cite{calvez2022,forien2022}, which we will refer to as ancestral lineage methodology, we demonstrate an alternative viewpoint on a spreading property of the population, known as the \textit{pushed} or \textit{pulled} dichotomy \cite{stokes1976}. Qualitatively, a \textit{pulled} wave is one in which the propagation of the wave results from the dynamics of the cells at the very leading edge of the wave, which are thus "pulling" the wave, whilst in a \textit{pushed} wave it is the whole population that contributes to the propagation of the wave. Moreover, these concepts are biologically relevant, as pushed waves promote genetic diversity inside the propagating
population, whereas pulled waves tend to select only the phenotypes at the leading edge
of the invasion \cite{roques2012}. The ancestral lineage methodology \cite{calvez2022,forien2022} aims at describing the backwards in time evolution of the relative position to the wave of a cell or its corresponding ancestor. In these latter works, it is shown that in the large population regime as $K\to+\infty$, these ancestral lineages converge to a simple diffusing particle with a drift term depending solely on the particle's. Whilst these works study cases involving \textit{pushed} waves, the application of this methodology to a model that exhibits a transition between a \textit{pulled} and \textit{pushed} regime, such as it is done in the present work, is new and sheds light on the difference of the dynamics of the ancestral lineages in these two regimes. In the \textit{pushed} regime, we observe that the ancestral lineage distribution converges to an equilibrium, which is consistent with the works in \cite{calvez2022,forien2022}. This equilibrium may be seen as a description of the contribution of each part of the wave to the propagation. On the contrary, in the \textit{pulled} regime, the ancestral lineages continuously drift to the right without ever reaching an equilibrium, selecting particles that lie further and further toward the leading edge of the population. Moreover, at the threshold of the transition from \textit{pulled} to \textit{pushed}, an interesting behavior is observed. Following \cite{an2021}, the spreading is designated as \textit{pushmi-pullyu}: There, the mean ancestral lineage drifts to the right, thus being consistent with a \textit{pulled} regime, yet, the mode of its distribution remains centered at the position of the $K$-th particle, which is a property shared with the \textit{pushed} regime.

\subsubsection*{Outline of the Paper}

In Section \ref{sec-presentation}, we give a detailed presentation of the stochastic particle system as well as state and discuss the main theorems and results presented in this work.

In Section \ref{sec:construction}, we briefly give a rigorous construction of the model and prove some preliminary properties that will be useful in the proofs of the following sections. 

In Section \ref{sec-convergence},  we prove that the limit of the empirical cumulative distribution function converges to a PDE governing the cumulative distribution function of a population in a continuous model. This extends the convergence results for rank-based systems in \cite{jourdain2000} to branching particles, where the branching mechanism also depends on the rank, as well as to systems with discontinuous drift coefficients. We do so by first proving tightness of the empirical measure, which follows standard procedures of many works on branching populations (see \textit{e.g.} \cite{FontbonaMeleard}), combined with certain technical amendments specific to the present problem.
The main point for us is that all the coefficients relative to the movement or the demographic events are bounded, and the difficulty lies in the identification of the limit, as despite being bounded, the coefficients are discontinuous. Switching to the empirical cumulative distribution smoothens the coefficients to become Lipschitz continuous, but not globally smooth which is the usual case in the previous works \cite{jourdain2013}.

In  Section \ref{sec-uniqueness}, the uniqueness of the limit is proven by applying a contraction principle to the position of the threshold mediating the switch between the \textit{Go} and \textit{Grow} regime. The uniqueness proof here is to an extent reminiscent of a fixed point strategy for existence and uniqueness in an analogous model \cite{demircigil2022}, but is based on different ideas used to describe the position of the threshold.

In Section \ref{sec-numerics}, we propose interesting numerical investigations that illustrate: 1. the spreading properties of the model and 2. the ancestral lineage methodology applied to the present model. First, we observe that the population spreads linearly, which justifies the relevance of the model for investigating traveling wave-like behaviors. Depending on the drift parameter $\chi$, we observe two different regimes of spreading, which are consistent with equivalent regimes for the large population limit studied in \cite{demircigil2022,demircigilhenderson}. Moreover, as $K$ increases, the spreading speed gets closer and closer to the analogous wave speed in the large population limit $\s^*=\left\{\begin{array}{ll}
     \chi + \frac{1}{\chi}& \chi>1  \\
     2&  \chi\leq 1
\end{array}\right.$. These observations on the spreading properties lead us to establish a conjecture, whose proof is left to future work. Then, by analogy of the results on the ancestral lineage dynamics in \cite{calvez2022,forien2022}, we conjecture that in the present model the ancestral lineages converge in the large-population regime ($K\to +\infty$) to a similar stochastic process, \textit{i.e.} a diffusing particle, whose evolution depends solely on its relative position within the wave. The consequences of this conjecture are illustrated numerically. In particular, the difference of the behaviors of the ancestral lineages in the \textit{pushed} ($\chi>1$) and \textit{pulled} ($\chi\leq1$) regimes  are investigated, as well as the threshold case ($\chi=1$), which will be referred to as \textit{pushmi-pullyu} ($\chi=1$) following \cite{an2021}, and which may be seen as a special subcase of a \textit{pulled} wave.

We finish this section by giving some notations that will be used throughout the whole paper.

\subsubsection*{Notations}
\begin{itemize}
    \item $\Mcal_F(\R) $
    denotes the set of finite measures on $\R$ endowed with the weak topology. $\Mcal$ denotes the subset of $\Mcal_F(\R)$ containing finite point measures : 
    $$\Mcal:= \left\{\sum_{i=1}^n \delta_{x_i}; \quad n \in \N , x_1, \dots, x_n \in \R\right\}.$$
    \item Similarly, we introduce the space of weighted finite point measures on $\R$ as
$$\Mcal_K:= \frac{1}{K}\Mcal=\left\{\frac{1}{K}\sum_{i=1}^n \delta_{x_i}; \quad n \in \N , x_1, \dots, x_n \in \R\right\}.$$
    \item $\DD([0,T), \Mcal_F(\R))$ denotes the Skorokhod space of left limited and right continuous functions from $[0,T)$ to $\Mcal_F(\R)$, endowed with the Skorokhod topology. 
    \item By default  $\DD([0,T], \Mcal_F(\R))$ is endowed with weak topology. When we want to emphasize it, we will write  $\DD([0,T], \Mcal^w_F(\R)))$.  We write $\DD([0,T], \Mcal^v_F(\R)))$ when  we endow $\Mcal_F(\R)$ with the vague topology.
    %\item $\JJ$ denotes an abstract index set.
    \item $\mathcal H: \R \to \R$ denotes the heavy side function, \textit{i.e. } $\mathcal H(x) = \mathbbm{1}(x\geq 0)$ .
    \item Let $\N^\ast = \N \setminus \{0\}$. Let $H^K = (H_1^K , \dots, H_k^K ,\dots) : \Mcal_K \to  (\R \cup \{-\infty\} )^{\N^\ast}$ be defined by
$
H^K \left(\frac{1}{K}\sum_{i=1}^n \delta_{x_i} \right) = (x_{\sigma(1)} , \dots, x_{\sigma(n)} , -\infty,-\infty, \ldots),$ where $\sigma$ is a permutation such that $x_{\sigma(1)}\geq x_{\sigma(2)}\geq \ldots \geq x_{\sigma(n)}$. Using a slight abuse of notation, we will simply drop the superscript $K$ and write $H=H^K$.
\item Individual labels can be chosen in the Ulam-Harris-Neveu set  $\mathcal U = \cup_{n\geq 1} \N^\ast \times \{1,2\}^n$.
\end{itemize}

\section{Presentation of the Model and of the Main Results}
\label{sec-presentation}

We first present in detail the stochastic model under consideration in Section~\ref{sec-individualbasedmodel} and we give one of the main statements about its properties (see Proposition~\ref{prop:mu-charact}). Then, in Section~\ref{subsec:mainresult}, we present the large population limit equation and its cumulative distribution function. The main result is given in Theorem~\ref{thm:mainresult} and concerns the convergence result of the cumulative empirical distribution. In Theorem~\ref{thm:uniqueness}, under additional assumptions on the initial data, the uniqueness of the limit is shown. Finally, we discuss in Section~\ref{subsec:numericalResPres} our numerical results on the speed of propagation of the particle system and on its ancestral lineages, differentiating between pushed and pulled fronts.

\subsection{Rank-Based Branching Particle System}
\label{sec-individualbasedmodel}

Let us now describe the dynamics our population. It will be modeled by a point measure valued process that will take into account the motion of the particles and the birth events. Fix $K\in N$ and define the space of weighted finite point measures on $\R$ as
$$\Mcal_K:= \left\{\frac{1}{K}\sum_{i=1}^n \delta_{x_i}; \quad n \in \N , x_1, \dots, x_n \in \R\right\}.$$
  The stochastic process $(\mu_t^K)_{t\geq 0}$ represents the spatial configuration of our particles weighted by $\frac{1}{K}$. That is, for $t\geq 0$,
  \begin{equation}
      \label{eq:muk}
      \mu_t^K= \frac{1}{K}\sum_{i\in V_t^K} \delta_{X_t^i},
  \end{equation}
  where $V_t^K$ is the set of labels of individuals alive at time $t$ and $X_t^i$ is the position of the individual with the label $i$. For a given initial state of the population $\mu_0^K$ with $N_0^K$ individuals (or atoms in $\mu_0^K$ ), we label
its atom by $1, \dots, N^K_0$. The offspring of an individual $u\in \mathcal{U}$ are denoted by $u1, u2$. Denoting by $N_t^K$ the size of the population at time $t$, it holds that $N^K_t= |V_t^K|= K \langle \mu_t^K,1 \rangle$.
  
  First, we will give an informal description of the dynamics of individuals. The initial population at $t=0$ is given by $\mu_0^K$. The general \textit{Go or Grow} idea is that the  particle population is divided into two subgroups as follows:
  \begin{itemize}
   \item (\textit{Grow}) The $K$ particles that are furthest to the right of the real line diffuse with a constant diffusion coefficient and may divide with rate equal to one. By division we mean that at the position of the particle that divides, a new particle is added to the system.
      \item (\textit{Go}) The rest of the particles diffuse,  have a constant drift coefficient $\chi>0$ and do not divide.
    
  \end{itemize}
Notice that the total rate of division is always equal to $K$. Hence to determine the instants in which new particles arrive to the system one may consider an exponential clock with  parameter $K$. Once that clock rings, one chooses uniformly among the $K$ particles on the far right and adds a new one in its place. \\
As for the positions  between the times of birth events $(T_m)_{m\geq 0}$, with $T_0=0$,  they evolve according to the following dynamics, for $t\in (T_m, T_{m+1}],$
\begin{equation}
    \label{eq:sde-positions}
    X_t^i= X^i_{T_m}+ \chi \int_{T_m}^t \mathbbm{1}_{( \sum_{j\in  V_{T_m}}\mathcal  H(X_s^j-X_s^i)>K)} ds + \sqrt{2}(W_t^i - W^i_{T_m}), \quad  i\in V_{T_m}.
\end{equation}
where $(W^i)_{i\in \mathcal{U}}$ are independent standard  one-dimensional Brownian motions defined on some probability space $(\Omega, \F, \PP)$. We notice here that once the labels  of individuals and the times $T_m$ are given, the SDE of type \eqref{eq:sde-positions} admits a unique strong solution, as the diffusion coefficient is constant and the drift bounded (see \cite{Veretennikov1982}). In addition, as the drift is bounded and as we are in $\R$, by Girsanov transformation particles may collide but will be instantaneously separated by their respective Brownian motions. We denote   the drift term of a particle at the position $x$ belonging to a population state $\nu\in \Mcal_K $  by   $\chi a(x,\nu)$, where
$$a(x,\nu):= \mathbbm{1}_{(\mathcal \langle  \nu , \mathcal{H}(\cdot-x)\rangle > 1)}.$$
Similarly, the individual birth rate of a cell at the position $x$ belonging to a population state $\nu\in \Mcal_K $   is   
$$b(x,\nu):= 1-a(x,\nu)=\mathbbm{1}_{(\mathcal \langle  \nu , \mathcal{H}(\cdot-x)\rangle \leq 1)}.$$
Clearly, $|a(\cdot,\cdot)|\leq 1$ and $|b(\cdot,\cdot)|\leq 1.$ 

We will construct rigorously  $(\mu_t^K)_{t\geq 0}$ and give some of its main features in Section~\ref{sec:construction}. Namely,  the following characterization of the process  $(\mu_t^K)_{t\geq 0}$ will be proven:
\begin{proposition}
\label{prop:mu-charact}
Assume $\EE\left[\langle \mu_0^K,1\rangle)\right]<\infty$. Then the process $(\mu_t^K)_{t\geq 0}$ given by \eqref{eq:muk} satisfies the following stochastic differential equation: For all $f \in C^{1,2}_b(\R_+ \times \R)$,  
     \begin{align}
     \nonumber
    &\langle \mu^K_t, f_t\rangle = \langle \mu^K_0, f_0\rangle + \int_0^t \left.\int \left( \frac{\partial f_s}{\partial s}(x) + \Lcal_{\mu^K_s} f_s(x) + b(x,\mu^K_{s})f_s(x) \right)   \mu^K_s (dx) ds\right. \\
    &+ M^{K,f, W}_t + M^{K,f, \Ncal}_t,
    \label{eq-sdemu}
\end{align}
where $\Lcal_\nu f(x):=\chi\frac{\partial f}{\partial x}(x)a(x,\nu)+\frac{\partial^2 f}{\partial x^2}(x)$, $ M^{K,f, W}$ is a continuous local martingale with quadratic variation
 \begin{align}
 \label{eqDefQuadVarW}
    \left \langle M^{K,f,W} \right\rangle_t= \int _0^t \frac{1}{K^2}  \sum_{ i\in V_s^K } \left(\frac{\partial f_s}{\partial x}(X^i_s)\right)^2 ds = \frac{1}{K} \int _0^t \left\langle \mu_s^K, \left(\frac{\partial f_s}{\partial x}\right)^2\right\rangle ds ,
 \end{align} 
and $ M^{K,f, \Ncal}$ is a pure jump local martingale with predictable quadratic variation 
%\begin{align*}
 %   &M^{K,\phi}_t:=\int_0^t \int_{\N^\ast} \int_0^1  \frac{1}{K}\mathbbm{1}_{\left( i\leq  \langle \mu^K_{s^-}, K  \rangle \right)  } \mathbbm{1}_{\left(\theta \leq b(H^i(\mu^K_{s^-}), \mu^K_{s^-})\right)} \phi(H^i(\mu^K_{s^-})) \Ncal(ds, di, d\theta) \\
  %  &- \int_0^t \int_{\N^\ast} \int_0^1  \frac{1}{K}\mathbbm{1}_{\left( i\leq  \langle \mu^K_{s^-}, K  \rangle \right)  } \mathbbm{1}_{\left(\theta \leq b(H^i(\mu^K_{s^-}), \mu^K_{s^-})\right)} \phi(H^i(\mu^K_{s^-})) ds \delta(di) d\theta
%\end{align*}
%is a martingale, satisfying:
 \begin{align}
 \label{eqDefQuadVarN}
    \left\langle M^{K,f, \Ncal} \right\rangle_t= \int_0^t  \frac{1}{K^2}\sum_{i\in V^K_s }b(X^i_{s}, \mu^K_{s}) f^2_s(X^i_s)   ds =\frac{1}{K} \int _0^t \langle \mu_s^K, b(\cdot, \mu^K_{s}) f^2_s \rangle ds. 
\end{align}
In addition, for $F^f: \Mcal_K \to \R$ of the form $F^f(\nu)= F(\langle \nu, f\rangle)$ where $f\in C^2(\R)$ and $F\in C^2_b(\R)$, the process 
$$F^f\left(\mu^K_t\right) - F^f\left(\mu^K_0\right) - \int_0^t L^K F^f\left(\mu_s^K\right) ds $$
is a local martingale where for $\nu \in \mathcal{M}_K$ we have the following expression for $L^K $:
\begin{align*}
    &L^K F^f(\nu) = K \int  b(x, \nu) \left(F^f\left(\nu + \frac{1}{K} \delta_x\right) - F^f(\nu)\right) \nu (dx) + \langle \nu, \Lcal_\nu f \rangle \left(F^f\right)'(\langle \nu, f \rangle )\\
    &+\frac{1}{K}\langle \nu,  (f')^2 \rangle  \left(F^f\right)''(\langle \nu, f \rangle ).
\end{align*}
\end{proposition}

\subsection{Limit PDE and the Convergence Result}
\label{subsec:mainresult}

Formally, as $K\to \infty$, the limit PDE  has the following form 
\begin{equation}
 \label{eq:limitPDE}
 \frac{\partial}{\partial t} u_t= \frac{\partial^2}{\partial x^2}u_t-\chi \frac{\partial}{\partial x}\left( u_t \tilde a\left(\int_x^{+\infty}u_t(y)dy\right) \right)+ u_t  \tilde b\left(\int_x^{+\infty}u_t(y)dy\right),
\end{equation}
with $\tilde a(z):=\mathbbm{1}_{z > 1}$ and $\tilde b(z):=1-\tilde a(z)$. Note that with this notation, we have $\tilde a(\langle \nu, \mathcal{H}(\cdot - x)\rangle)=a(x,\nu)$ and $\tilde b(\langle \nu, \mathcal{H}(\cdot - x)\rangle)=b(x,\nu)$. The difficulty in taking the limit here stems from the fact that the coefficients $\tilde a, \tilde b$ are discontinuous. In order to circumvent that difficulty, similarly to the work in \cite{jourdain2000}, we turn to the cumulative distribution, which will satisfy an equation leading to higher regularity.
  
  Let us denote $F_t(x)= \int_x^\infty u_t(y)dy$ and write the equation for $F$. It comes
  \begin{equation}
 \label{eq:limitPDE-F}
 \frac{\partial}{\partial t} F_t= \frac{\partial^2}{\partial x^2}F_t-\chi \frac{\partial}{\partial x}A (F_t) + B(F_t),
\end{equation}
where $A(z):= \int_0^z\tilde a(s)ds= (z-1) \vee 0$ and $B(z):= \int_0^z\tilde b(s)ds= z-A(z)=z \wedge 1$. For the above equation, we use the following notion of weak solution:

\begin{definition}
\label{def:F-weak}
    A continuous mapping $F: t\in [0, \infty) \to F_t \in L^1_{loc}(\R)$ is a weak solution to \eqref{eq:limitPDE-F} if  for all $\varphi_t\in C_c^{1,2}([0, \infty) \times \R)$ we have
\begin{align}
\label{eq:weak-sol}
    \int \varphi_t(x) & F_t(x) dx - \int \varphi_0(x) F_0(x) dx= \int_\R \int_0^t  F_s(x) \frac{\partial}{\partial s}\varphi_s(x) \ ds \ dx \nonumber\\
    & + \int_\R \int_0^t \left\{ F_s(x) \partial_x^2 \varphi(s,x) + \chi A(F_s(x)) \partial_x \varphi_s(x) + B(F_s) \varphi_s(x)\right\} ds \ dx.
\end{align}
\end{definition}

Similarly, we give a notion of weak solution for Equation (\ref{eq:limitPDE}):

\begin{definition}
\label{def:u-weak}
    A continuous mapping $u: t\in [0, \infty) \to u_t \in L^1_{loc}(\R)\cap L^1(\R_+)$ is a weak solution to \eqref{eq:limitPDE-F} if  for all $\varphi_t\in C_c^{1,2}([0, \infty) \times \R)$ we have
\begin{align}
\label{eq:weak-sol_u}
    \int \varphi_t(x) & u_t(x) dx - \int \varphi_0(x) u_0(x) dx= \int_\R \int_0^t  \left\{u_s(x) \frac{\partial}{\partial s}\varphi_s(x) +u_s(x) \partial_x^2 \varphi(s,x)\right\}\ ds \ dx \nonumber\\
    & + \int_\R \int_0^t \left\{  \chi  \tilde{a}\left(\int_x^{+\infty}u_s(y)dy\right)u_s  \partial_x \varphi_s(x) + \tilde{b}\left(\int_x^{+\infty}u_s(y)dy\right)u_s \varphi_s(x)\right\} ds \ dx.
\end{align}
\end{definition}

Let us note here that $u_t \in L^1_{loc}(\R)\cap L^1(\R_+)$ means that the function $u_t$ is integrable at $x=+\infty$, which is necessary in order to properly define $F_t(x)=\int_x^{+\infty}u_t(y)dy$, and thus be able to properly define the \textit{Go} and the \textit{Grow} regimes.

For a given $\varphi_t \in C_c^{1,2}([0, \infty) \times \R)$, we define $\phi_t(x)=\Hcal \ast \varphi_t =\int_{-\infty}^x\varphi_t(y)dy$. We denote $\tilde \Hcal = \Hcal(-\cdot)$ and recall that $\langle \Hcal \ast f, g \rangle = \langle f, \tilde \Hcal \ast g \rangle $, as well as $\langle \tilde \Hcal \ast f, g \rangle = \langle f,  \Hcal \ast g \rangle $.
        Then we define for $m\in \DD([0,T], \Mcal_F(\R))$ the operator:
        \begin{align}
        \nonumber
            &\Fcal_{t,\varphi}(m)=\left\langle \tilde\Hcal \ast m_t,\varphi_t\right\rangle-\left\langle \tilde\Hcal \ast m_0,\varphi_0\right\rangle-\int_0^t\left\langle \tilde\Hcal \ast m_s,\partial_s \varphi_s\right\rangle ds\\
            &-\int_0^t\left\langle \tilde\Hcal \ast m_s,\partial_x^2 \varphi_s\right\rangle ds-\int_0^t\left\langle \chi A(\tilde\Hcal \ast m_s),\partial_x\varphi_s\right\rangle ds-\int_0^t\left\langle B(\tilde\Hcal \ast m_s),\varphi_s\right\rangle ds 
            \label{eq-defF}
        \end{align}
The first main result of this paper is that the cumulative distribution of $(\Hcal \ast \mu_t^K)_{t\leq T}$ converges, along subsequences and in probability, to a weak solution to \eqref{eq:limitPDE-F}.

\begin{theorem}
\label{thm:mainresult}
    Let $T>0$  and  $(\mu_t^K)_{t\leq T}$, which satisfies Equation~(\ref{eq-sdemu}). Suppose 
    \begin{equation}
        \label{cond:initialdata}
        \EE\left[\left|\left\langle \mu^K_0, 1\right\rangle\right|^3\right]<+\infty
    \end{equation}
    %$\EE\left[|\langle \mu^K_0, 1\rangle\right]|<+\infty$ \textcolor{red}{+ convergence of the initial condition} 
    and that $\mu_0^K$ converges , as $K\to \infty$, in probability in $ \Mcal_F^w(\R)$ to some $\mu_0$.
    Denote $(Q^K)_{K\geq 1}=\mathcal L(\mu^K)_{K\geq 1}$. Then, 
    \begin{enumerate}
        \item The sequence $(Q^K)_{K\geq 1}$ is tight in $\Pcal(\DD([0,T], \Mcal_F(\R)))$ w.r.t weak topology.
        \item Consider an accumulation point $Q^\infty\in \Pcal(\DD([0,T], \Mcal_F(\R)))$ of the sequence $(Q^K)_{K\geq 1}$ and $\mu^\infty$ in the support of $Q^\infty$. Then $\Fcal_{t,\varphi}(\mu^\infty)=0$, where $\Fcal_{t,\varphi}$ is given by (\ref{eq-defF}), and $(\tilde{\mathcal{H}}\ast \mu_t^\infty)_{t\leq T}$ is a weak solution to Equation (\ref{eq:limitPDE-F}).
    \end{enumerate}
\end{theorem}
The proof is presented in Section~\ref{sec-convergence} and it goes through a tightness-consistency argument. The interesting point is the identification of the limit, where we used the discrete integration by parts inspired by \cite{jourdain2000,jourdain2014}.

As a consequence of the latter theorem, we have that any function $F$ in the support of $\tilde{\mathcal{H}}\ast \mu_t^\infty$ is a weak solution of Equation (\ref{eq:weak-sol}) and satisfies $F\in L^\infty([0,T]\times \R)$ (see preamble of Section \ref{sec-uniqueness}). For this class of functions, we then establish the following regularity result for weak solutions of Equation (\ref{eq:weak-sol}):

\begin{proposition}
\label{prop:reg}
    [Regularity of $F$ solution to Equation (\ref{eq:weak-sol}) and $u:=-\partial_x F$ solution to Equation (\ref{eq:weak-sol_u})]
    Suppose that the initial datum $\mu_0$ is a deterministic function $\mu_0=u_0\in L^\infty(\R)\cap L^1(\R)$ and consider $F\in L^\infty([0,T]\times \R)$ a weak solution to Equation (\ref{eq:weak-sol}) with initial datum $F_0(x)=\int_x^{\infty}u_0(y)dy$. Then,
    \begin{align*}
        F\in C([0,T], W^{1,\infty}(\R)).
    \end{align*}
    Moreover, $u:=-\partial_x F$ is a solution to Equation (\ref{eq:weak-sol_u}) and satisfies for all $p\in [0,\infty]$:
    \begin{align*}
        u\in C([0,T], L^p(\R)).
    \end{align*}
\end{proposition}

With this regularity result, we are able to establish a uniqueness theorem, using some additional assumptions on the initial datum. Thus, under these assumptions, Theorem \ref{thm:mainresult} provides not just convergence along sub-sequences, but also actual convergence to a unique function:
\begin{theorem}
\label{thm:uniqueness}
    [Uniqueness of Weak Solutions]
    Suppose that the initial condition $\mu_0$ is a deterministic function $\mu_0=u_0$ and satisfies the following conditions:
    \begin{enumerate}
        \item[i)] $u_0 \in L^1(\R)\cap L^\infty(\R)$
        \item[ii)] $u_0 \geq 0$
        \item[iii)] there exists $\bar{x}(0)\in \R$ such that $F_0(\bar{x}(0))=\int_{\bar{x}(0)}^{+\infty}u_0(x) dx = 1$
        \item[iv)] there exists $\eta >0$ and $\underline{u}>0$ such that for all $x\in (\bar{x}(0)-\eta, \bar{x}(0)+ \eta)$, $u_0(x)\geq \underline{u}$
    \end{enumerate}
    Then, a weak solution $F\in L^\infty([0,T]\times \R)$ of Equation (\ref{eq:limitPDE-F}) satisfying the initial condition $F_0(x)=\int_x^{\infty}u_0(y)dy$  is unique.
\end{theorem}
Let us briefly comment on the additional assumptions iii) and iv) on the initial datum, which are specific to Theorem \ref{thm:uniqueness}. Assumption iii) describes the case, where initially some particles are in the \textit{Go} regime, or equivalently not all the particles are in the \textit{Grow} regime. The finite population analog ($K<+\infty$) would be that the initial number of particles is higher than $K$. Assumption iv) states that in a an open set containing the threshold $\bar{x}(0)$ (\textit{i.e.} the large population equivalent of the position of the $K$-th particle) $u_0$ is bounded below by a positive constant.

In Subsection \ref{sec-uniqueness-reg}, Proposition \ref{prop:reg} is proven by showing first that a weak solution $F\in L^\infty([0,T]\times \R)$ of Equation (\ref{eq:limitPDE-F}) is a mild solution that can be represented through Duhamel's Formula. Then, using standard estimates on the heat operator $e^{t\partial_x^2}$ and given the hypothesis $u_0 \in L^1(\R)\cap L^\infty(\R)$ and $u_0 \geq 0$, we show that $F\in C([0,T],W^{1,\infty}(\R))$. From this we are able to deduce a representation formula for $u$ and that $u\in C([0,T],L^p(\R))$ for all $p \in [1,\infty]$. 

Then in Subsection \ref{sec-uniqueness-unique}, under the additional assumptions on the initial condition $u_0$, listed in Theorem \ref{thm:uniqueness}, we show uniqueness of the weak solution. We start by using the regularity result of Proposition \ref{prop:reg} on $u$ to show that on a small time interval $t\in [0,T]$, $u_t$ remains bounded below by a positive constant in the neighborhood of the threshold $\bar{x}(t)$ (\textit{i.e.} $F_t(\bar{x}(t))=1$) at least locally in time. By integrating two solutions $u_{1,t},u_{2,t}$ over the spatial set, where $u_{1,t}$ and $u_{2,t}$ are not in the same (\textit{Go} or \textit{Grow}) regime, we are able to apply a contraction principle, leading to uniqueness locally in time, that is then through an elementary argument extended to all time.

\subsection{Speed of Propagation and Ancestral Lineages}
\label{subsec:numericalResPres}

The model investigated in this work leads to a deterministic system in the large population limit as $K\to+\infty$, a setting that has been explored in previous studies \cite{demircigil2022,demircigilhenderson}. Its strength lies in its elementary nature, which allows for a rigorous analysis of various spreading phenomena, while still exhibiting a rich and diverse range of behaviors, particularly influenced by the parameter $\chi$.

We start by recalling some of the properties of the deterministic limit model, established in \cite{demircigil2022,demircigilhenderson}, including the traveling waves and their speeds, the asymptotic spreading of the threshold $\bar{x}(t)$ (\textit{i.e.} $F_t(\bar{x}(t))=1$) and the classification of the traveling waves as either pushed or pulled through the methodology of neutral fractions \cite{garnier2012,roques2012}: In the case $\chi> 1$, the traveling wave is pushed, whilst in the case $\chi \leq 1$, the wave is pulled.

\begin{figure}[t]
\begin{center}
\includegraphics[width=0.5\linewidth]{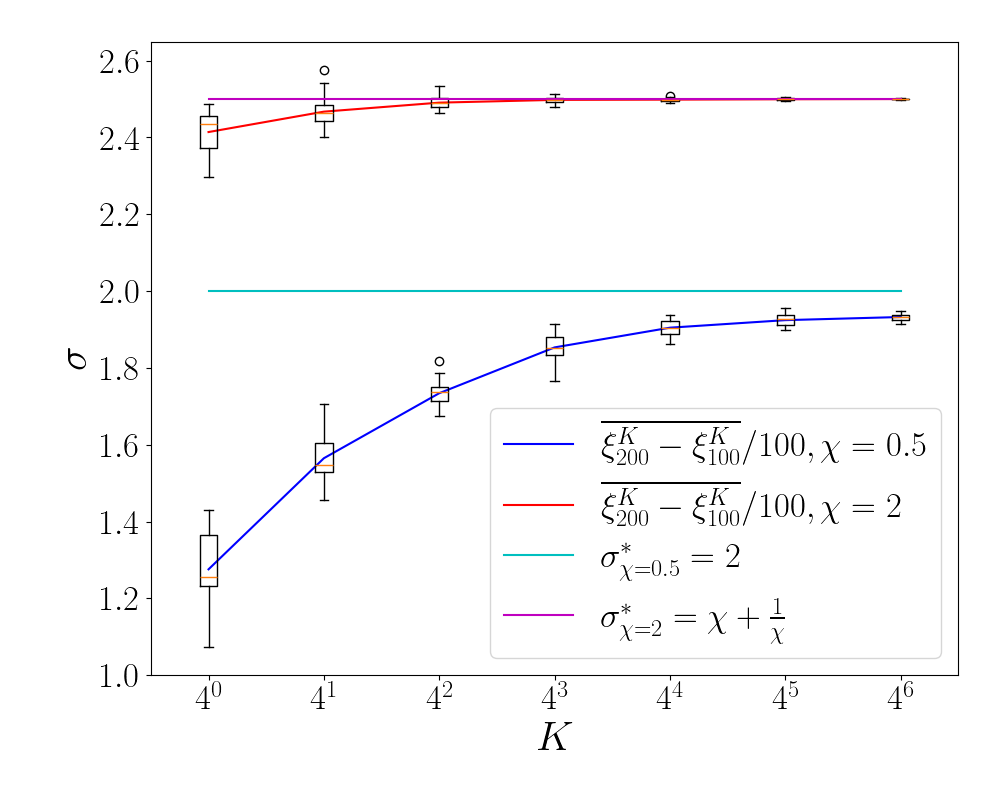}
\caption{The average velocity of the $K$-th particle over the interval $[100,200]$, \textit{i.e.} $\frac{\xi^K_{200}-\xi^K_{100}}{100}$, where $\xi^K_t=H_K(\mu^K_t)$ is the position of the $K$-th particle at time $t$, serves as an estimator of the expansion speed of the population. As $K$ increases, in average this estimator gets closer and closer to the deterministic traveling wave speed $\sigma^*=\left\{\begin{array}{ll}
    \chi+\frac{1}{\chi} & \text{ if }\chi>1 \\
    2 & \text{ if }\chi\leq 1
\end{array} \right.$ (see Figure \ref{fig:speed} for more details).}
\label{fig:intro-speed}
\end{center}
\end{figure}

\begin{figure}[b]
\begin{center}
\begin{subfigure}{0.5\linewidth}
  \centering
\includegraphics[width=\linewidth]{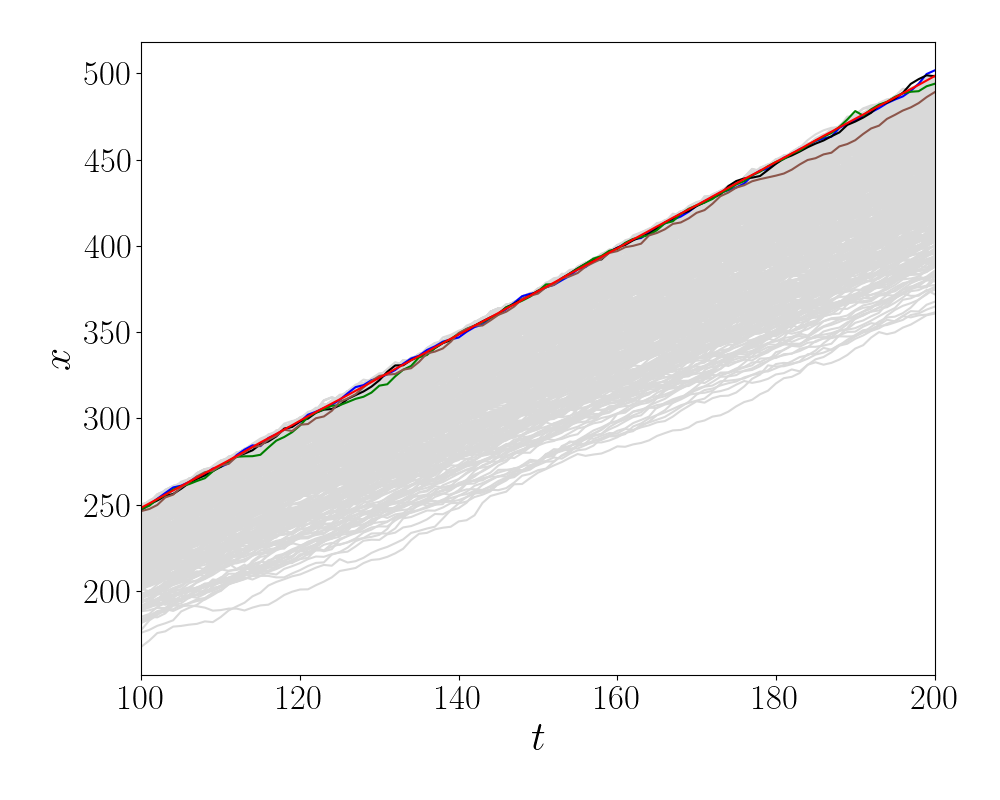}
  \caption{}
  \label{fig:intro-traj:sfig1}
\end{subfigure}%
\begin{subfigure}{7.5cm}
  \centering
\includegraphics[width=\linewidth]{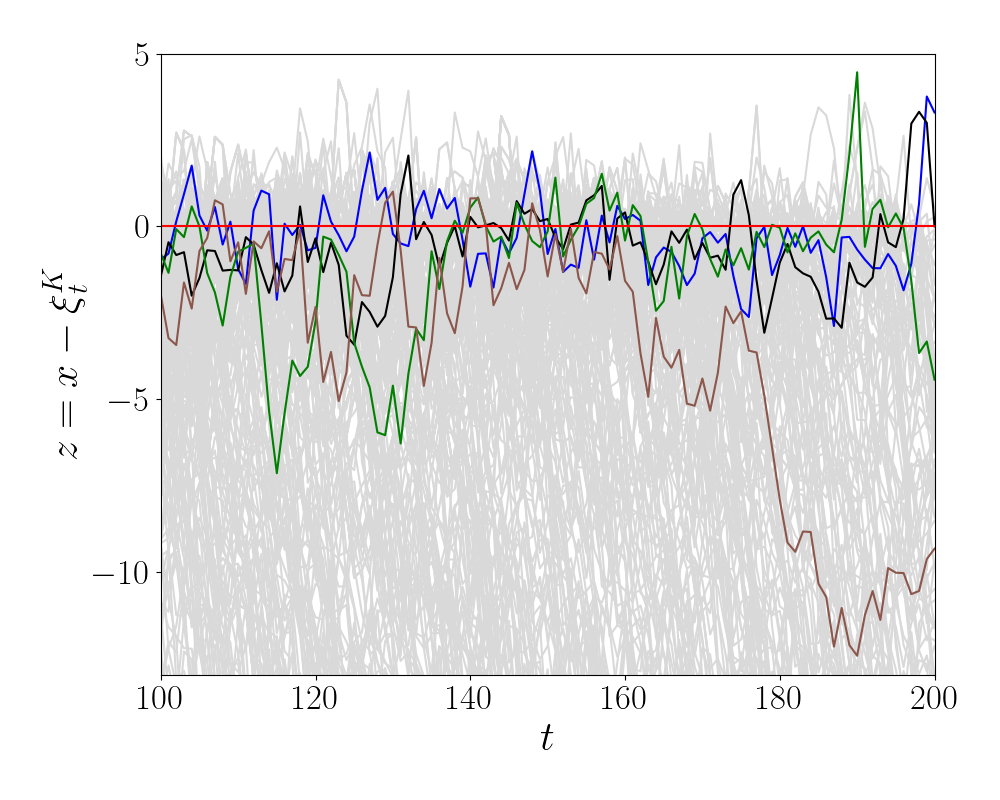}
  \caption{}
  \label{fig:intro-traj:sfig2}
\end{subfigure}
\caption{Graphical representation of the particle trajectories in (a) the stationary frame $(t,x)$, and (b) the moving frame $(t,z)=(t,x-\xi^K_t)$, where $\xi^K_t$ is the position of the $K$-th particle at time $t$. The red curve represents the position of the $K$-th particle and in (a) we observe a linear evolution of its position. Four more trajectories have been arbitrarily highlighted for illustration purposes and represent the ancestral lineage that evolves backwards in time, \textit{i.e.} that should be read from right to left (see Figure \ref{fig:trajectories} for more details). }
\label{fig:intro-traj}
\end{center}
\end{figure}

The goal of this section is, then, to numerically explore analogous properties of the given stochastic model. These numerical investigations lead us to propose two open problems, which we leave for future work. The first, a direct consequence of our numerical investigations, concerns the emergence of traveling waves in our model. The second one, which focuses on ancestral lineages, is inspired by recent literature \cite{calvez2022,forien2022} and supported by our simulations. What sets our application of the ancestral lineage methodology apart is that we apply it to a model exhibiting a transition from pulled to pushed waves.

First, we start by simulating the process for different values of $K$ and show that the population spreads linearly in the sense that in average the position of the particle of rank $K$ evolves linearly. As $K$ increases (see Figure \ref{fig:intro-speed}), the average velocity of the particle of rank $K$ gets closer and closer to the traveling wave speed $\sigma^*=\left\{\begin{array}{ll}
    \chi+\frac{1}{\chi} & \text{ if }\chi>1 \\
    2 & \text{ if }\chi\leq 1
\end{array} \right.$ for the deterministic model, with a slight discrepancy in the case $\chi\leq 1$ that may be explained through the so-called Bramson shift \cite{bramson1978,demircigilhenderson}. These observations lead us to conjecture the following behavior:

\begin{conjecture}
\label{conjecture:TW}
Let $\mu_0^K$ be an initial distribution that is compactly supported.
\begin{enumerate}
    \item Linear spreading in the finite population regime. There exists a constant propagation speed $\s^K$ such that in an appropriate convergence sense, $\lim_{t\to+\infty} \frac{\xi^K(t)}{t}=\s^K$, where $\xi^K_t=H_K(\mu^K_t)$ is the position of the $K$-th particle at time $t$.
    \item Local convergence to a traveling wave profile in the moving frame. Let $\Kcal\subset \R$ be a compact subset, then $\hat{\mu}^{K}_t:=\mu^K_t(\cdot -\xi^K(t))$ converges in law to a stationary distribution $\hat{\mu}^{K}_{\infty}$ on $\Kcal$.
    \item Convergence of the finite population propagation speed to the deterministic propagation speed.  $\lim_{K\to+\infty} \s^K = \s^*$, where $\sigma^*=\left\{\begin{array}{ll}
    \chi+\frac{1}{\chi} & \text{ if }\chi>1 \\
    2 & \text{ if }\chi\leq 1
\end{array} \right.$.
\end{enumerate}
\end{conjecture}

\begin{figure}[t]
\begin{center}
\begin{subfigure}{.5\linewidth}
  \centering
\includegraphics[width=\linewidth]{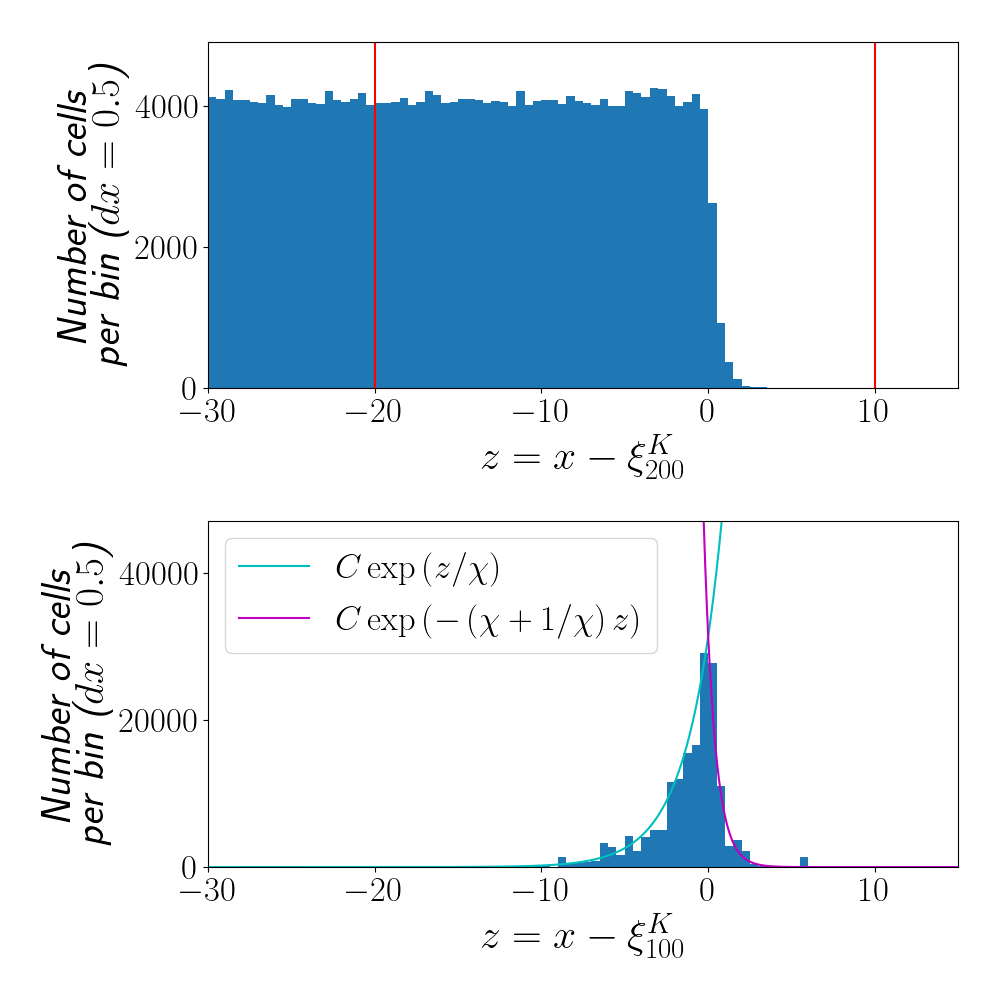}
  \caption{}
  \label{fig:intro-AL:sfig1}
\end{subfigure}%
\begin{subfigure}{0.5\linewidth}
  \centering
\includegraphics[width=\linewidth]{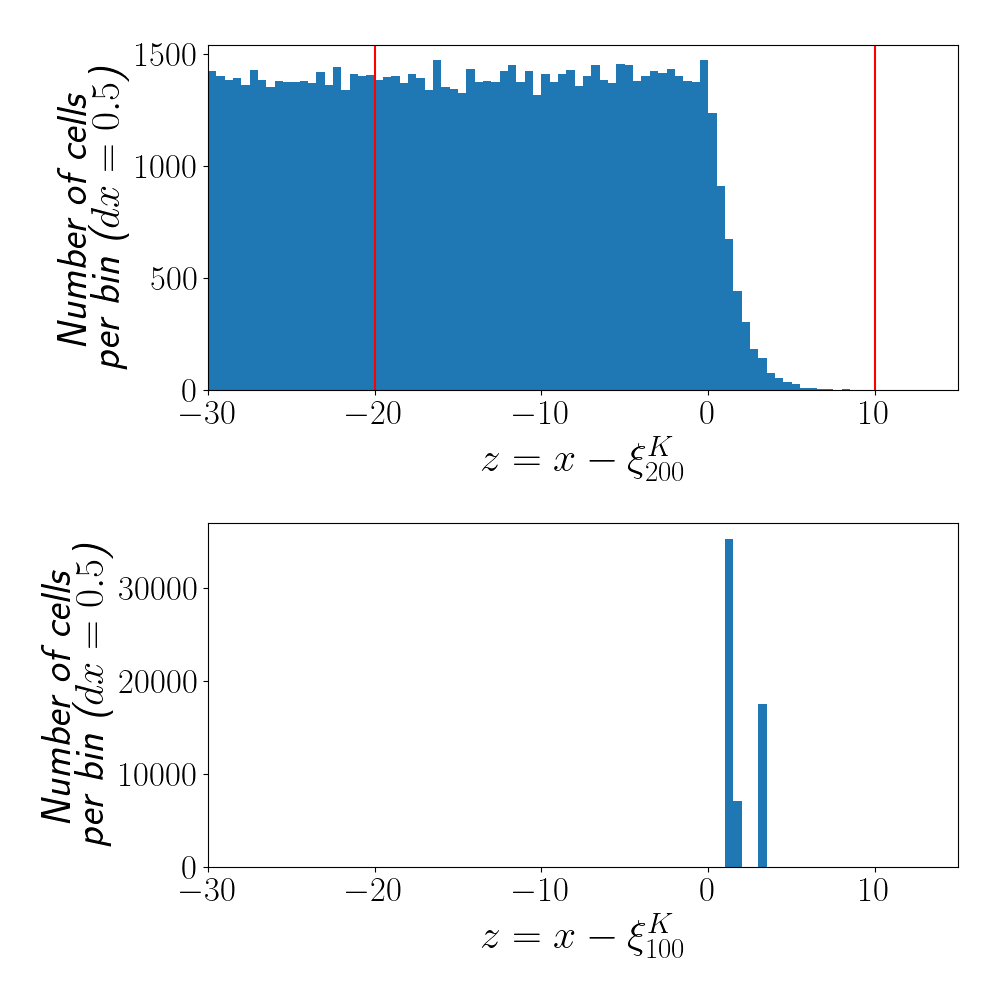}
  \caption{}
  \label{fig:intro-AL:sfig2}
\end{subfigure}
\caption{
Graphical representation of the ancestral lineage distribution. (a): The pushed case with $\chi=2$. (b): The pulled case with $\chi=\frac{1}{2}$. The top figure represents the histogram of $ \mu_T^K$ in the frame $z=x-\xi^K_T$ for $T=200$, $K=4096$. The distribution of $ \mu_T^K$ is qualitatively close to the deterministic traveling wave $u^{\s^*}$ (see also Figure \ref{fig:histogram}) The two red bars stake out the positions of the particles selected, whose ancestral lineages are tracked, which are all the particles whose position is in $[-20,10]$.  The bottom figure represents the position of the ancestral lineages in the frame $z=x-\xi^K_t$ at time $t=100$, \textit{i.e.} $s=100$, of the selected particles (which are between red bars in the top figure). The minimum of the cyan and the magenta curves represents the predicted equilibrium of Equation (\ref{introSDE}) in the case $\chi>1$. (see Figure \ref{fig:backward-distribution} for more details).}
\label{fig:intro-AL}
\end{center}
\end{figure}

Then, we investigate the pulled or pushed nature of the expansion in the stochastic model, following the methodology of ancestral lineages proposed independently in \cite{calvez2022,forien2022}. Whilst in these works the investigated waves were pushed, in the case of our model the methodology of ancestral lineages will be illustrated on a pulled to pushed transition.

The ancestral lineage refers to the evolution of a particle (or its unique ancestor alive at the considered moment in time) backwards in time. More precisely, given the position at time $t$ of a particle $X^i_t$ with $i\in V_t$ (the set of indices of particles alive at time $t$), we denote $Y^{K}_{s,t}$ its position or the position of its ancestor alive at time $ t-s$ for $s\geq 0$, \textit{i.e.} $Y^{K}_{s,t}=X^j_{t-s}$, where $j$ is the unique prefix of $i$ in $V_{t-s}$. $\hat{Y}^K_{s,t}$ denotes the position of the ancestral lineage in the moving frame that centers at $z=0$ the position of the $K$-th particle, \textit{i.e.} $\hat{Y}^K_{s,t}={Y}^K_{s,t}-\xi^K_{t-s}$. Ancestral lineage trajectories $\hat{Y}^K_{s,t}$ in the moving frame are highlighted in Figure \ref{fig:intro-traj:sfig2} and should be read from right to left. 

Then, we consider an initial distribution $\mu^K_0$ that is given by the deterministic traveling wave profile $u^{\s^*}$ associated with the traveling wave speed $\s^*$ (see \cite{demircigil2022,demircigilhenderson}). By definition of this traveling wave, $\mu^\infty_t$, the large population limit $K\to+\infty$, will merely be a shifted version of $u^{\s^*}$ for further time $t\geq 0$  (by a shift of $\s^* t$).

Following the results in \cite{calvez2022,forien2022}, we conjecture that, even though in the finite population case ($K<+\infty$), the ancestral lineage clearly depends on the precise position of all other particles, in the large population limit $K\to +\infty$, the ancestral lineage converges to a diffusing particle, whose evolution depends solely on its relative position within the wave, or more precisely:

\begin{conjecture}
\label{conjecture:AL}
Let the initial distribution $\mu_0^K$ be given by a deterministic traveling wave profile $u^{\s^*}$ and let $\hat{Y}^K_{s,t}$ be an ancestral lineage process in the frame of the traveling wave. Then, in the large population limit $K\to+\infty$,  $\hat{Y}^K_{s,t}$ converges in distribution to a process $\hat Y_{s,t}$ that is a solution to the following SDE:
\begin{align}
\label{introSDE}
    d\hat Y_{s,t} = \beta\left(\hat Y_{s,t}\right)ds + \sqrt{2}dW_s,
\end{align}
where $\beta(z)=\sigma^*-\chi \mathbbm{1}_{z<0}+2\frac{\partial_z u^{\s^*}}{u^{\s^*}}$.
\end{conjecture}

As noted, the drift term $\beta$ is solely a function of the relative position within the wave and it can be computed explicitly, as the explicit wave profile $u^{\s^*}$ is known (see Theorem \ref{thm-tw} in Section \ref{sec-numerics}, originally from \cite{demircigil2022,demircigilhenderson}).

The conjectured behavior of the ancestral lineage is illustrated numerically by simulating $\hat{Y}^K_{s,t}$ for large but finite $K$. In addition, two different regimes of behaviors of $\hat{Y}^K_{s,t}$ depending on the parameter $\chi$ are observed: In the case $\chi>1$, as $s$ decreases, the sampled distribution of $\hat{Y}^K_{s,t}$ appears to converge to an equilibrium very close to the predicted equilibrium of Equation (\ref{introSDE}) (see Figure \ref{fig:intro-AL}); whilst in the case $\chi\leq 1$, the sampled distribution of $\hat{Y}^K_{s,t}$ drifts indefinitely to the right, as $s$ decreases. These different behaviors are consistent with Equation (\ref{introSDE}) as will be shown and indicate how the methodology of ancestral lineage may serve as an alternative interpretation for the dichotomy between \textit{pushed} and \textit{pulled} waves:
\begin{itemize}
    \item in the \textit{pulled} case $\chi\leq 1$, we observe numerically that the mean of the probability distribution $v(s)$ of the process $\hat{Y}_{s,t}$ drifts to progressively higher values of $z$, meaning that the ancestors at earlier times were located at the leading edge of the traveling wave and the ancestors
    correspond to a very small number of particles that "pull" the rest of the population (see Figure \ref{fig:intro-AL:sfig2}).
    \item in the \textit{pushed} case $\chi>1$, it is proven that the probability distribution $v(s)$ of the process $\hat{Y}_{s,t}$ given by Equation (\ref{introSDE}) converges to an equilibrium, describing the typical position of ancestors in the traveling wave (see Figure \ref{fig:intro-AL:sfig1}). The term \textit{pushed} expresses then the fact that the typical position of ancestors is located in the bulk of the wave. Thus, those particles are "pushing" the wave, as opposed to the previous case. 
    \item in the critical case $\chi=1$ (sometimes called the \textit{pushmi-pullyu} case \cite{an2021}), it is shown that the distribution's mean drifts indefinitely to the right, sharing thus a feature of pulled waves, whilst simultaneously keeping its mode centered at $z=0$, which is a feature of pushed waves.
\end{itemize}

\section{Construction and Preliminary Properties}
\label{sec:construction}

The process  $(\mu_t^K)_{t\geq 0}$ can be rigorously constructed as a Markov process solving an SDE driven by independent Brownian motions $(W^i)_{i \in \mathcal{U}}$ and a Poisson point measure $\mathcal{N}(ds, di, d\theta) $ on $\R^+ \times \mathcal{U} \times [0,1]$ with intensity $ds\times \left(\sum_{j\in \mathcal{U}} \delta_j(di) \right)\times d\theta$ independent of the Brownian motions and all independent of the initial condition $\mu_0^K$. Let us  give here an explicit algorithmic construction in terms of these objects. 
\begin{enumerate}
    \item Set $T_0=0$ and $m=0$. Denote $N_0^K= K\langle \mu^K_0,1\rangle$  and label the atoms of $\mu_0^K$ by $1, \dots , N_0^K$. Define accordingly $V_0^K$, the set of labels of individuals alive at initial time.
    \item The population follows the dynamics 
    \begin{equation}
    \label{eq:PS}
      X_t^i= X^i_{T_m}+ \chi \int_{T_m}^t a(X^i_s, \mu^K_s)ds + \sqrt{2}(W_t^i - W^i_{T_m}), \quad  i\in V_{T_m}.  
    \end{equation}
    until the next jump time $t>T_m$ of the measure $\mathcal{N}(dt, di, d\theta)$ such that the corresponding atom $(t,i, \theta)\in \R^+ \times \mathcal{U}\times [0,1]$ satisfies 
    $$i \in V^K_{T_m}, \quad \theta \leq  b(X^i_t, \mu^K_t). $$
    \item Set $T_{m+1}=t$, $N_{T_{m+1}}= N_{T_{m}}+1$. The population $\mu^K_{T_{m+1}}$ is updated in the following way: One atom in $X^i_{t^-}$ is replaced by 2 atoms in the same position denoted by $X^{i1}_t$ and $X^{i2}_t$, $i$ is removed from $V^K_{T^-_{m+1}}$, 
    and $i1$ and $i2$ are added,\textit{ i.e.} 
    $$V^K_{T_{m+1}}= (V^K_{T_{m}} \setminus \{i\}) \cup \{i1,i2\}.$$
    \item Set $m=m+1$ and go to step 2.
\end{enumerate}
\begin{lemma}
    Assume $\EE(\langle \mu_0^K,1\rangle)<\infty$. Then the process $(\mu_t^K)_{t\geq 0}$ given by \eqref{eq:muk} is well defined. 
\end{lemma}
\begin{proof}
    Notice that $(T_m)_{m\geq 1}$ are stopping times w.r.t. the filtration generated by the Brownian motions  $(W^i)_{i \in \mathcal{U}}$ and $\Ncal(ds, di, d\theta)$ and the random variable $\mu_0^K$. Conditionally to $\F_{T_m}$, the drift of the particle system \eqref{eq:PS} is bounded and so $(X^i_t)_{i\in V_{T_m}}$ can be defined as the restriction on $(T_m, T_{m+1}]$ of the particle system with fixed number of particles solving the same equation on $(T_m, \infty)$. Hence, $\mu_t^K$ is well defined for $t\leq T_\infty$, where $T_\infty= \lim_{m\to \infty} T_m$. It remains to prove that $T_\infty= \infty$ a.s. This is trivial here, as for any $m\geq 1$, conditionally to  $\F_{T_m}$,  $T_{m+1}-T_m$ has exponential distribution with a fixed parameter $K\wedge N_{T_m}^K$. Indeed, taking for example $m=1$, we have that for $t>0$ $$\PP(T_1\geq t)= \PP (N^K_t= N^K_0).$$
    The latter means that point measure $\mathcal{N}$ did not produce any atom on $(0,t]$. Hence, $$\PP(T_1\geq t)= e^{-\int_0^t   \int_{\mathcal{U}} \int_0^1 \mathbbm{1}_{\left( i\in V_0^K  \right)  } \mathbbm{1}_{\left(\theta \leq b(X^i_{s^-}, \mu^K_{s^-})\right)} ds \times \left(\sum_{j\in \mathcal{U}} \delta_j(di) \right)\times d\theta}$$
$$= e^{-\int_0^t   \sum_{i\in V_0^K}  b(X^i_{s^-}, \mu^K_{s^-}) ds }.$$
Noticing that for any $s\geq 0$ when $N_0\geq K$ one has $\sum_{i\in V_0^K}   b(X^i_{s^-}, \mu^K_{s^-})  =K$ gives 
$$\PP(T_1\geq t)=  e^{-Kt}.$$
If the initial condition is such that $N_0<K$, then the waiting time of the first event has $\mathcal{E}(N_0)$ distribution. In the same way, the waiting time $T_2-T_1$ of the second event  has  $\mathcal{E}(N_0+1)$ distribution, etc. Until we reach the $(K-N_0+1)$-th event whose waiting time will have $\mathcal{E}(K)$ distribution and the same will be true for all events after that one. As starting from some time, the waiting time between the events admits $\mathcal{E}(K)$ distribution, it must be that $T_\infty=\infty$ a.s.
\end{proof}

Let us now derive the stochastic differential equation for $\mu^k$ and prove Proposition \ref{prop:mu-charact}. 
\begin{proof}[Proof of Proposition~\ref{prop:mu-charact}]
    For $\nu \in \Mcal_K$ and $\phi \in C_2(\R)$ we introduce the differential operator related to the dynamics of the positions of our cells, 
$$\Lcal_\nu \phi (x) := \chi \frac{\partial\phi}{\partial x} (x) a(x, \nu) + \frac{\partial^2\phi}{\partial x^2} (x). $$   Let $(t,x)\mapsto f_t(x)\in C^{1,2}_b(\R_+ \times \R)$. By construction, $t\mapsto \langle \mu^k_t, f_t\rangle$ is right continuous with limits on the left. Informally speaking, in order to derive the equation satisfied by  
$\langle \mu^k_t, f_t \rangle$ one decomposes it into increments according to the jumps $(T_m)_{m\geq 0}$ and applies It\^ o's formula on each interval. Around the jump times, one also picks up the terms in which  $f_t$ is evaluated in the admissible atoms produced by the measure $\mathcal{N}$.

Rigorously, we apply It\^ o's formula on $(X_s^i)_{i\in V^K_{T_m}}$ for $s \in [T_m, T_{m+1})$, then we sum over $i\in V^K_{T_m}$ and divide by $K$ and, finally, for a given $t>0$ we sum over all $m\in \N$ such that $\{t> T_m\}$. It comes 
\begin{align}
    &\langle \mu^K_t, f_t\rangle = \langle \mu^K_0, f_0\rangle +  \int_0^t \int \Big( \frac{\partial f_s}{\partial s}(x) + \Lcal_{\mu^K_s} f_s(x) \Big)   \mu^K_s (dx) ds \\
    & +  \frac{\sqrt{2}}{K}\int_0^t \sum_{ i\in V_s^K } \frac{\partial f_s}{\partial x}(X^i_s)  dW_s^i + \int_0^t \int_\mathcal{U} \int_0^1  \mathbbm{1}_{\left( i\in V_s^K, \theta \leq b(X^i_{s^-} , \mu^K_{s^-})\right)} \frac{f_s(X^i_{s^-})}{K} \Ncal(ds, di, d\theta). \nonumber
\end{align}
Compensating the Poisson measure, we may rewrite the above expression in order to obtain Equation (\ref{eq-sdemu}), which we recall:
 \begin{align*}
    &\langle \mu^K_t, f_t\rangle - \langle \mu^K_0, f_0\rangle -  \int_0^t \int \left( \frac{\partial f_s}{\partial s}(x) + \Lcal_{\mu^K_s} f_s(x) + b(x,\mu^K_{s})f_s(x) \right)   \mu^K_s (dx) ds  \\ 
    &= M^{K,f, W}_t + M^{K,f, \Ncal}_t,
\end{align*}
where $ M^{K,f, W}$ is a continuous local martingale with quadratic variation exactly as in~\eqref{eqDefQuadVarW}
and $ M^{K,f, \Ncal}$ is a pure jump local martingale with predictable quadratic variation as in~\eqref{eqDefQuadVarN}.
From here, it has become standard (see \textit{e.g.} \cite{FontbonaMeleard}) to characterize the law of $\mu^K$ using the cylindrical functions $F^f: \Mcal \to \R$ of the form $F^f(\nu)= F(\langle \nu, f\rangle)$ where $f\in C^2(\R)$ and $F\in C^2_b(\R)$. In fact, 
rewrite the zero order term above as 
\begin{align*}
    \int b(x,\mu^K_{s})f(s,x) \mu^K_s (dx)&= K \int \left\langle \frac{1}{K}\delta_x, f\right\rangle b(x,\mu^K_{s})\mu^K_s (dx)\\
    &= K \int \left(\left\langle \mu_s^K+ \frac{1}{K}\delta_x, f\right\rangle - \left\langle \mu_s^K, f\right\rangle  \right) b(x,\mu^K_{s})\mu^K_s (dx)
\end{align*}

By It\^o's formula the process 
$$F^f(\mu^K_t) - F^f(\mu^K_0) - \int_0^t L^K F^f(\mu_s^K) ds $$
is a local martingale where for $\nu \in \mathcal{M}^K$ we have the following expression for $L^K $:
\begin{align*}
    &L^K F^f(\nu) = K \int  b(x, \nu) \left(F^f\left(\nu + \frac{1}{K} \delta_x\right) - F^f(\nu)\right) \nu (dx) + \langle \nu, \Lcal_\nu f \rangle \left(F^f\right)'(\langle \nu, f \rangle )\\
    &+\frac{1}{K}\langle \nu,  (f')^2 \rangle  \left(F^f\right)''(\langle \nu, f \rangle ).
\end{align*}
\end{proof}

We end this section with the following statement.
\begin{lemma}
\label{lemma:massUniform}
    Let $t>0$ and $K\geq 1$. Then, assuming $\EE \left[\langle \mu_0^K, 1\rangle \right]< \infty$, 
    \begin{equation}
        \label{eq:masscontrol}
        \EE \left[\sup_{t\leq T}  \langle \mu_t^K, 1\rangle \right]\leq\EE \left[\langle \mu_0^K, 1\rangle \right] + T,
    \end{equation}
    and, assuming for $p>1$ that $\EE\left[\langle \mu_0^K, 1\rangle)^p\right]$, we have
    \begin{equation}
        \label{eq:masscontrolp}\EE \left[\sup_{t\leq T}  (\langle \mu_t^K, 1\rangle ^p)\right]\leq C_p (\EE\left[\langle \mu_0^K, 1\rangle)^p\right] + T^p +(T/K)^{p/2}). 
    \end{equation}
\end{lemma}
\begin{proof}
Recall that $N_t^K$ denotes the size of the population at time $t>0$. As for any $s>0$ we have that $\sum_{i\in V_s^K} b(X_s^i, \mu^K_s) = K \wedge N_s^K$ a.s., testing with $f\equiv  1 $ in \eqref{eq:muk}, the following holds:
 $$\langle \mu_t^K, 1\rangle=\langle \mu_0^K, 1\rangle +\frac{1}{K}\int_0^t\sum_{i\in V_s^K} b(X_s^i, \mu^K_s) ds +M_t^{K,1, \mathcal{N}}=\langle \mu_0^K, 1\rangle +\frac{1}{K}\int_0^t K \wedge N_s^K ds +M_t^{K,1, \mathcal{N}} $$
 Hence, by BDG inequality, for $p\geq 1$ we have that 
 $$\EE \left[\sup_{t\leq T}  (\langle \mu_0^K, 1\rangle )^p\right]\leq C_p\left( \EE\left[\langle \mu_0^K, 1\rangle)^p\right] + T^p +(T/K)^{p/2}\right) $$
 and of course
    $$\EE \left[\langle \mu_t^K, 1\rangle \right]= \EE  \left[\langle \mu_0^k, 1\rangle \right] + \frac{1}{K}\int_0^t\EE  \left[\sum_{i\in V_s^K} b(X_s^i, \mu^K_s)\right] ds\leq \EE \left [\langle \mu_0^K, 1\rangle \right] + t. $$
\end{proof}

\section{Cumulative Distribution Limit }
\label{sec-convergence}

\subsection{Tightness}

First, we will show the tightness of  $(\mu^K)_{K\geq 1}$ seen as a random sequence in $\DD([0,T], M^v_F(\R)))$. 
  For a given $K\geq 1$, 
 denote by $\tilde Q^K \in \Pcal(\DD([0,T], M^v_F(\R)))$  the law of $\mu^K$. 
  
According to Roelly-Copolleta  \cite[Theorem 2.1]{roelly1986} the sequence of laws $(\tilde Q^K)_{K\geq 1}$  is tight in $\Pcal(\DD([0,T], M^v_F(\R)))$   if the following two conditions hold:
\begin{itemize}
    \item for any $f\in \mathcal{C}_0^2(\R)$ we have that $\left((\langle \mu_t^K,f\rangle)_{t\in[0,T]}\right)_{K\geq 1}$ is tight in $\DD([0,T], \R)$,
    \item  the above statement holds also when $f\equiv 1$.
\end{itemize}

To check the above conditions, we use the same argument for any $f\in \mathcal{C}_0^2(\R)$ and $f\equiv 1$. By Aldous-Rebolledo criterion  (see \textit{e.g.} \cite{joffe1986}), the desired tightness follows from the following two conditions:
\begin{enumerate}
    \item For every $t\leq T$, $\left(\langle \mu_t^K,f\rangle\right)_{K\geq 1}$ is tight in $\R$,
    \item For every  $\varepsilon>0$, there exists $\delta>0$  and $K_0$  such that for every sequence of stopping times $(\tau_K)_{K\geq 1}$ we have
    $$\sup_{K\geq K_0}\sup_{\theta\in [0,\delta]}
    \PP \left[|\langle \mu^K_{\tau_K+\theta}-\mu^K_{\tau_K},f \rangle|>\varepsilon\right]<\varepsilon.
    $$
\end{enumerate}

    Let $t\leq T$. For the first point, we notice that  \eqref{eq:masscontrol}  and the fact that  all the coefficients in \eqref{eq-sdemu} are bounded imply that
    $$\sup_{K}\EE\left[ |\langle \mu^K_t, f\rangle |\right]<C_{f,T} .
        $$
        Hence, by Markov inequality 
        $$ \sup_K\mathbb P [\langle \mu^K_t,f\rangle >n] \leq \frac{C_{f,T}}{n}.$$
        Passing to the limit as $n\to \infty$  concludes the first point. 
        
 The second point  follows from the martingale decomposition of our process in (\ref{eq-sdemu},\ref{eqDefQuadVarW},\ref{eqDefQuadVarN}) and boundedness of the coefficients.  Indeed,  for $\delta>0$ and $\theta \in [0, \delta]$ we have
 \begin{align*}
    & \EE|\langle \mu^K_{\tau_K+\theta}-\mu^K_{\tau_K},f \rangle|\leq C_{f}\delta + \sqrt{\EE [ |M^{K,f ,W}_{\tau_K+\theta}-M^{K,f ,W}_{\tau_K}|^2]}+\sqrt{\EE [ |M^{K,f ,\mathcal{N}}_{\tau_K+\theta}-M^{K,f ,\mathcal{N}}_{\tau_K}|^2]}\\
    &\leq  C_{f}\delta + \sqrt{\EE [ (\langle M^{K,f ,W}\rangle_{\tau_K+\theta}-\langle M^{K,f ,W}_{\tau_K}\rangle)]}+\sqrt{\EE [ (\langle M^{K,f ,\mathcal{N}}\rangle_{\tau_K+\theta}-\langle M^{K,f ,\mathcal{N}}\rangle_{\tau_K} ]}\\
    & \leq  C_{f}\delta + C_f \sqrt{ \delta}.
 \end{align*}
Hence the Markov inequality again gives the desired result choosing $\delta$ small enough.
 This gives us the tightness of $\mu^K$ in $\DD([0,T], M^v_F(\R)))$.
 
 For a (non-relabeled) subsequence of $\mu^K$, denote the limit law by $\tilde \Pi$ on   $\DD([0,T], \Mcal_F(\R)))$ and the canonical process in $\DD([0,T], \Mcal_F(\R)))$ by $\mu$. We now wish to show the tightness holds also for $\DD([0,T], M^w_F(\R)))$.
 To do so, 
 according to Roelly-Meleard~\cite[Theorem 3]{meleard1993}, we have to show that:
\begin{enumerate}
\item tightness of $\mu^K$ in $\DD([0,T], M^v_F(\R)))$ holds,
    \item the canonical process under $\tilde \Pi$ is continuous, that is that it belongs to $C([0,T], M^w_F(\R)))$,
    \item the sequence of laws of the processes 
 $(\langle \mu^K_t,1\rangle)_{t\leq T}$ converges weakly to the law of the process  $(\langle \mu_t,1\rangle)_{t\leq T}$
\end{enumerate}
We already obtained the first point. We will come back to the second point later. 

For the third point, define $\Psi: \R \to \R$ of class $\mathcal{C}^2$ non decreasing on $\R^+$  and such that $\Psi(x)= 1$ for $|x|\geq 2$, $\Psi(x)= 0$ for $|x|\leq 1$. For any $m\geq 1$, denote by $\Psi_m(x)= \Psi\left(\frac{|x|}{m}\right)$. Let $F : \DD ([0,T], \R)\to \R$ be bounded and  Lipschitz function. We use below the notation $\langle \mu^K_\cdot ,1\rangle= (\langle \mu^K_t ,1\rangle)_{t\leq T}$. We will show that 
$$\lim_{K\to \infty} \EE \left[F(\langle  \mu^K_\cdot ,1\rangle)\right]=\EE [ F(\langle \mu_\cdot ,1\rangle)].$$
We decompose for $m\geq 1$ in the following way:
\begin{align*}
  & \big| \EE [F(\langle  \mu^K_\cdot ,1\rangle)]-\EE [ F(\langle \mu_\cdot ,1\rangle)]\big |\leq \big| \EE [F(\langle  \mu^K_\cdot ,1\rangle)]-\EE [ F(\langle \mu^K_\cdot ,1-\Psi_m\rangle)]\big | \\
   &+  \big| \EE [F(\langle  \mu^K_\cdot ,1-\Psi_m\rangle)]-\EE [ F(\langle \mu_\cdot ,1-\Psi_m\rangle)]\big | +  \big| \EE [F(\langle  \mu_\cdot ,1-\Psi_m\rangle)]-\EE [ F(\langle \mu_\cdot ,1\rangle)]\big |
\end{align*}
By the Lipshitz  continuity of $F$ and the fact that the Skhorohod distance on $\DD ([0,T], \R)$ is bounded by the supnorm, we have 
\begin{align*}
  & \big| \EE [F(\langle  \mu^K_\cdot ,1\rangle)]-\EE [ F(\langle \mu_\cdot ,1\rangle)]\big |\leq  C \EE \left[\sup_{t\leq T}\big|\langle  \mu_t^K ,1\rangle-\langle \mu_t^K ,1-\Psi_m\rangle\big |\right] \\
   &+  \big| \EE [F(\langle  \mu^K_\cdot ,1-\Psi_m\rangle)]-\EE [ F(\langle \mu_\cdot ,1-\Psi_m\rangle)]\big | + C \EE \left[\sup_{t\leq T}\big|\langle  \mu_t ,1\rangle-\langle \mu_t ,1-\Psi_m\rangle\big |\right]\\
   & \leq C \EE\left[ \sup_{t\leq T}\big|\langle  \langle \mu_t^K ,\Psi_m\rangle\big |\right] + C \EE \left[ \sup_{t\leq T}\big|\langle  \langle \mu_t ,\Psi_m\rangle\big |\right]\\
   &+  \big| \EE [F(\langle  \mu^K_\cdot ,1-\Psi_m\rangle)]-\EE [ F(\langle \mu_\cdot ,1-\Psi_m\rangle)]\big |.
\end{align*}
We will prove that the first two terms on the r.h.s. are arbitrary small in $m$ uniformly in $K$. For the last term, as for fixed $m$, the function $1-\Psi_m \in C^2_c$, by convergence in the vague topology, that term is uniformly small in $K$.
It remains to obtain that 
\begin{equation}
    \label{eq:mukpsim}
    \lim_{m\to \infty }\sup_K \EE \left[\sup_{t\leq T} \langle \mu_t^K, \Psi_m \rangle\right ]=0
\end{equation}

and 
\begin{equation}
    \label{eq:mupsim}
\lim_{m\to \infty } \EE \left[\sup_{t\leq T}\langle \mu_t, \Psi_m \rangle \right]=0
\end{equation}

Applying \eqref{eq-sdemu} with $\Psi_m$ we have that
\begin{align*}
&\langle \mu_t^K, \Psi_m \rangle = \langle \mu_0^K, \Psi_m \rangle +   \int_0^t \int \Big(   \Lcal_{\mu^K_s} \Psi_m(x) + b(x,\mu^K_{s})\Psi_m(x) \Big)   \mu^K_s (dx) ds\\
& + M^{K,\Psi_m, W}_t + M^{K,\Psi_m, \Ncal}_t,
\end{align*}
where 
$$    \left\langle M^{K,\Psi_m,W} \right\rangle_t + \left\langle M^{K,\Psi_m, \Ncal} \right\rangle_t = \frac{1}{K} \int _0^t \left\langle \mu_s^K, \left(\frac{\partial \Psi_m}{\partial x}\right)^2 + b(\cdot, \mu^K_{s}) \Psi_m^2 \right\rangle ds.$$
{Notice that $\norm{\dx \Psi_m }_\infty\leq \frac{\norm{\dx \Psi}_\infty}{m}$ and that $\norm{\dx^2 \Psi_m }_\infty\leq \frac{\norm{\dx^2 \Psi}_\infty}{m^2}$.} Hence,
$$\Lcal_{\mu^K_s} \Psi_m(x) \leq \frac{C_{\Psi,\chi}}{m}.   $$
The latter combined with the fact  that $a$ and $b$ are bounded yields
\begin{align*}
    \sup_{t\leq T} \langle \mu_t^K, \Psi_m \rangle &\leq \langle \mu_0^K, \Psi_m \rangle + \frac{C}{m} \int_0^T \langle \mu^K_s,1\rangle ds +   \int_0^T \sup_{s\leq t} \langle    \Psi_m, \mu^K_s  \rangle ds\\
    & + \sup_{t\leq T} M^{K,\Psi_m, W}_t +\sup_{t\leq T } M^{K,\Psi_m, \Ncal}_t.
\end{align*}
Applying BDG inequality and Lemma \ref{lemma:massUniform} leads to
\begin{align*}
&\EE\left[\sup_{t\leq T} \langle \mu_t^K, \Psi_m \rangle\right ] \leq \EE \left[\langle \mu_0^K, \Psi_m \rangle \right]+ \frac{C_T+ \EE  \left[\langle \mu^K_0,1\rangle\right]}{m} \\
&+ \int_0^T  \EE  \left[\sup_{s\leq t} \langle    \Psi_m, \mu^K_s  \rangle\right] dt + \frac{C}{K} \int_0^T \EE\left[ \langle \mu^K_t,1\rangle\right] dt. 
\end{align*}
 Gronwall inequality and again Lemma \ref{lemma:massUniform} give
$$\EE\left[\sup_{t\leq T} \langle \mu_t^K, \Psi_m \rangle \right] \leq \left(\EE \left[\langle \mu_0^K, \Psi_m \rangle \right]+ \frac{C_T+ \EE \left[ \langle \mu^K_0,1\rangle\right]}{m} +  \frac{CT}{K} \left(\EE\left[ \langle \mu^K_0,1\rangle\right] +T \right)\right)e^T . $$
Taking the supremum in $K$ and the limit $m\to \infty$  we obtain that \eqref{eq:mukpsim} holds provided that $\lim_{m\to \infty }\sup_K \EE \langle \mu_0^K, \Psi_m \rangle=0 $. Given the convergence assumption on $\mu_0^K$ and the moment condition \eqref{cond:initialdata}, the latter holds (see \textit{e.g.} \cite[Lemma 3.4]{FontbonaMeleard}, \cite[Equation (32)]{MelTran}).
%(\textcolor{red}{This is done in \cite{FontbonaMeleard} p. 13}).

Now, from here we prove \eqref{eq:mupsim} through a monotone convergence argument. Define for  $l\in \N$, the functions  $\Psi_{m,l}:= \Psi_m (1-\Psi_{l}) $. These are compactly supported functions that are monotonically increasing to $\Psi_m$ as $l \to \infty$. Then, by the convergence in vague topology of $\mu^K$ and the fact that $\Psi_{m,l} \leq \Psi_{m}$  we have that
$$\EE\left[sup_{t \leq T} \langle \mu_t ,   \Psi_{m,l}\rangle\right] = \lim_{K \to \infty }\EE\left[\sup_{t \leq T} \langle \mu_t^K ,   \Psi_{m,l}\rangle\right] \leq \liminf_{K \to \infty} \EE\left[\sup_{t \leq T} \langle \mu_t^K ,   \Psi_{m}\rangle\right].$$
Now we take the limit as $l\to \infty$ and we use monotone convergence 
$$\EE\left[\sup_{t \leq T} \langle \mu_t ,   \Psi_{m}\rangle\right] = \lim_{K \to \infty }\EE\left[\sup_{t \leq T} \langle \mu_t^K ,   \Psi_{m,l}\rangle\right] \leq \liminf_{K \to \infty} \EE\left[\sup_{t \leq T} \langle \mu_t^K ,   \Psi_{m}\rangle\right].$$
We can conclude that  \eqref{eq:mupsim} holds combining the above inequality and \eqref{eq:mukpsim}.

Now we prove the second point. By the construction of the process $\mu^K$, it holds almost surely that
        $$\sup_{t\in [0,T]}\sup_{\norm{f}_\infty\leq 1}\left| \langle\mu^K_t,f \rangle-\langle\mu^K_{t^-},f \rangle\right|\leq \frac{1}{K}.
        $$
        
This implies that $\mu$ is a.s.  in $C([0,T], \Mcal_F^v(\R))$. A consequence of \eqref{eq:mupsim} is that we can extract a sequence from $\left(\sup_{t\leq T}\langle \mu, \Psi_m\rangle \right)_{m\geq 1}$ that a.s. converges to zero. Now, for $t>0$ and $f \in C_b(\R)$  passing to the limit when $m\to \infty$ in the following inequality 
$$|\langle \mu_t-\mu_{t^-}, f\rangle| \leq |\langle \mu_t-\mu_{t^-}, f(1-\Psi_m)\rangle| +  |\langle \mu_t, f\Psi_m\rangle| +  |\langle \mu_{t^-}, f\Psi_m\rangle| $$
we can conclude that $\mu \in C([0,T], \Mcal_F^w(\R))$ a.s.

\subsection{Identification of the Limit}

Before proving this result we need some technical preparation. We recall that for a given function $\varphi_t\in C^{1,2}_c([0,+\infty)\times \R)$, we define $\phi_t(x)=\Hcal \ast \varphi_t =\int_{-\infty}^x\varphi_t(y)dy $ and $\phi_t \in C^{1,2} _b([0,+\infty)\times \R)$, which makes $\phi_t$ an acceptable test function for the SDE (\ref{eq-sdemu}). For the sake of rewriting the SDE in a convenient manner, we now give a series of identities:
\begin{align}
\label{eqIBPphit}
    & \langle \mu^K_t, \phi_t\rangle = \langle \mu^K_t, \Hcal \ast\varphi_t\rangle = \langle \tilde \Hcal \ast  \mu^K_t , \varphi_t\rangle \\
\label{eqIBPphi0}
    &\langle \mu^K_0, \phi_0\rangle =  \langle \tilde \Hcal \ast  \mu^K_0 , \varphi_0\rangle \\
\label{eqIBPdsphi}
    & \int_0^t\left\langle \mu^K_s, \frac{\partial \phi_s}{\partial s}\right\rangle ds =   \int_0^t\left\langle \tilde \Hcal \ast  \mu^K_s , \frac{\partial \varphi_s}{\partial s}\right\rangle ds\\
\label{eqIBPdxxphi}
    & \int_0^t\left\langle \mu^K_s, \frac{\partial^2 \phi_s}{\partial x^2}\right\rangle ds =   \int_0^t\left\langle \tilde \Hcal \ast  \mu^K_s , \frac{\partial^2 \varphi_s}{\partial x^2}\right\rangle ds.
\end{align}
Using a similar argument than Jourdain (Lemma 1.5 in \cite{jourdain2000}), we apply an integration by parts to the nonlinear advection and birth terms. Indeed, consider:
\begin{align*}
    &\left\langle\mu_s^K, \frac{\partial \phi_s}{\partial x}a(\cdot,\mu_s^K)\right\rangle\\
    =&\frac{1}{K}\sum_{i=1 }^{\langle \mu^K_s,K \rangle } \frac{\partial \phi_s}{\partial x}(H^i(\mu^K_s))  a(H^i(\mu^K_s), \mu^K_s)\\
    =& \frac{1}{K}\sum_{i=1 }^{\langle \mu^K_s,K \rangle } \int_{-\infty}^{H^i(\mu^K_s)}\frac{\partial \varphi_s}{\partial x}(x)dx  a(H^i(\mu^K_s), \mu^K_s) \\
     =& \int_\R \frac{\partial \varphi_s}{\partial x}(x)\frac{1}{K}\sum_{i=1 }^{\langle \mu^K_s,K \rangle } \mathbbm{1}_{\left(x\leq H^i(\mu^K_s)\right)}  a(H^i(\mu^K_s), \mu^K_s) dx\\
     =& \int_\R \frac{\partial \varphi_s}{\partial x}(x)\sum_{i=1 }^{\langle \mu^K_s,K \rangle } \mathbbm{1}_{\left(x\leq H^i(\mu^K_s)\right)}  \frac{1}{K} \tilde a\left( \frac{i}{K}\right) dx,
\end{align*}
where we recall that $\tilde a(p)=\mathbbm{1}_{p>1}$. Observing next that $\tilde a(p)$ is constant on the intervals of the form $\left(\frac{i-1}{K},\frac{i}{K} \right]$:
\begin{align*}
    \frac{1}{K} \tilde a\left( \frac{i}{K}\right) = 
    \int_\frac{i-1}{K}^\frac{i}{K}\tilde a\left( \frac{i}{K}\right)dp = \int_\frac{i-1}{K}^\frac{i}{K}\tilde a\left(p\right)dp = \int_0^{+\infty}\tilde a\left(p\right)\1_{(\frac{i-1}{K}<p\leq \frac{i}{K})}dp .
\end{align*}
Let us remark here that, because of the preceding observation, we do not have to deal with a remainder term, unlike in \cite{jourdain2000}, where $\tilde a$ is not piecewise constant.

This leads to:
\begin{align*}
    & \int_\R \frac{\partial \varphi_s}{\partial x}(x)\sum_{i=1 }^{\langle \mu^K_s,K \rangle } \mathbbm{1}_{\left(x\leq H^i(\mu^K_s)\right)}  \frac{1}{K} \tilde a\left( \frac{i}{K}\right) dx\\
    =& \int_\R \int_0^{+\infty }\frac{\partial \varphi_s}{\partial x}(x)\tilde a(p)\sum_{i=1 }^{\langle \mu^K_s,K \rangle } \mathbbm{1}_{\left(x\leq H^i(\mu^K_s)\right)}  \1_{(\frac{i-1}{K}<p\leq \frac{i}{K})} dpdx
\end{align*}
We now show that:
\begin{align}
\label{eqConvIndicatrix}
    \sum_{i=1 }^{\langle \mu^K_s,K \rangle } \mathbbm{1}_{\left(x\leq H^i(\mu^K_s)\right)}  \1_{(\frac{i-1}{K}<p\leq \frac{i}{K})} =
    \1_{\left(p\leq \tilde \Hcal \ast \mu_s^K(x) \right)}
\end{align}
Consider $i_0$ the unique index such that $H^{i_0+1}\left(\mu^K_s\right)<x\leq H^{i_0}\left(\mu^K_s\right)$, we start by showing  that the sum in the left hand-side of Equation (\ref{eqConvIndicatrix}) is equal to $1$ if and only if $p\leq \frac{i_0}{K}$. Indeed, as the intervals $\left(\frac{i-1}{K}, \frac{i}{K} \right]$ are disjoint, at most one of the terms in the sum can be nonzero. Denote $j_0$, satisfying $\frac{j_0-1}{K}<p\leq \frac{j_0}{K}$, and which corresponds to the index of the only possibly nonzero term in the sum. If that term is equal to 1, then $x\leq H^{j_0}\left(\mu^K_s\right)$. But by the definition of $i_0$, this means that $ H^{i_0}\left(\mu^K_s\right)\leq H^{j_0}\left(\mu^K_s\right)$, which in turn implies that $j_0\leq i_0$ and that $p\leq \frac{j_0}{K}\leq \frac{i_0}{K}$. Conversely, if that term is equal to 0, then $x> H^{j_0}\left(\mu^K_s\right)$, which shows that $H^{j_0}\left(\mu^K_s\right)\leq H^{i_0+1}\left(\mu^K_s\right)$ and that $j_0\geq i_0+1$. This in turn implies that $p> \frac{j_0-1}{K}\geq \frac{i_0}{K}$. In conclusion, the sum in the left hand-side of Equation (\ref{eqConvIndicatrix}) is equal to $\1_{\left(p\leq\frac{i_0}{K}\right) }$. 
But $i_0$ precisely satisfies $\frac{i_0}{K}= \tilde \Hcal \ast \mu_s^K(x) $, which establishes Equation (\ref{eqConvIndicatrix}). Hence:
\begin{align*}
   & \int_\R \int_0^{+\infty }\frac{\partial \varphi_s}{\partial x}(x)\tilde a(p)\sum_{i=1 }^{\langle \mu^K_s,K \rangle } \mathbbm{1}_{\left(x\leq H^i(\mu^K_s)\right)}  \1_{(\frac{i-1}{K}<p\leq \frac{i}{K})} dpdx \\
    =&\int_\R \int_0^{+\infty }\frac{\partial \varphi_s}{\partial x}(x)\tilde a(p) \1_{(p\leq \tilde \Hcal \ast \mu^K(x) )}dpdx\\
    =& \int_\R \frac{\partial \varphi_s}{\partial x}(x)A\left( \tilde \Hcal \ast \mu^K(x)\right)dx
\end{align*}
Thus in conclusion, we have that:
\begin{align}
\label{eqIBPA}
\left\langle \mu_s^K,\frac{\partial \phi_s}{\partial x}a(\cdot,\mu_s^K)\right\rangle = \left\langle\mu_s^K, \Hcal\ast\frac{\partial \varphi_s}{\partial x}a(\cdot,\mu_s^K)\right\rangle
    = \left\langle A( \tilde \Hcal \ast \mu_s^K ), \frac{\partial \varphi_s}{\partial x}\right\rangle 
\end{align}
Similarily:
\begin{align}
\label{eqIBPB}
\left\langle \mu_s^K,\phi_sb(\cdot,\mu_s^K)\right\rangle=
    \frac{1}{K} \sum_{i=1 }^{\langle \mu^K_s,K \rangle } \phi_s(H^i(\mu^K_s)) b(H^i(\mu^K_s), \mu^K_s) 
    %=& \int_\R \varphi_s(y)\sum_{i=1 }^{\langle \mu^K_s,K \rangle } \mathbbm{1}_{\left(y\leq H^i(\mu^K_s)\right)}  \frac{1}{K} b\left( \frac{i}{K}\right) dy \\
    = \langle B( \tilde \Hcal \ast \mu_s^K ), \varphi_s \rangle.
\end{align}
Next, for a fixed $\varphi\in C^{1,2}_c([0,+\infty)\times\R)$ we recall the definition (\ref{eq-defF}) of the functional $\Fcal_{t,\varphi}$ for $m\in \DD([0,T], \Mcal_F(\R))$ by:
 \begin{align*}
            &\Fcal_{t,\varphi}(m)=\langle \tilde\Hcal \ast m_t,\varphi_t\rangle-\langle \tilde\Hcal \ast m_0,\varphi_0\rangle-\int_0^t\langle \tilde\Hcal \ast m_s,\partial_s \varphi_s\rangle ds\\
            &-\int_0^t\langle \tilde\Hcal \ast m_s,\partial_x^2 \varphi_s\rangle ds-\int_0^t\langle \chi A(\tilde\Hcal \ast m_s),\partial_x\varphi_s\rangle ds-\int_0^t\langle B(\tilde\Hcal \ast m_s),\varphi_s\rangle ds 
        \end{align*}

Our goal is to establish that the support of $\tilde\Hcal \ast \mu^\infty $ consists of weak solutions to Equation (\ref{eq:limitPDE-F}), which is equivalent to:
\begin{align}
\label{eq:FcalMuInftyIsZero}
    \Fcal_{t,\varphi}(\mu^{\infty}) = 0\hspace{.5cm} \Pi^\infty(d\mu^\infty)\text{ a.s.}
\end{align}

\textbf{Step a):} We prove that  $\lim_{K\to +\infty} \EE\left[\left|\Fcal_{t,\varphi}(\mu^K) \right|^p \right] = 0$.

Using SDE (\ref{eq-sdemu}) and Identities (\ref{eqIBPphit}, \ref{eqIBPphi0}, \ref{eqIBPdsphi}, \ref{eqIBPdxxphi}, \ref{eqIBPA}, \ref{eqIBPB}), we observe that:
\begin{align*}
    \Fcal_{t,\varphi}\left( \mu^K \right) = M_t^{K,\phi,W}+ M_t^{K,\phi,\Ncal}
\end{align*}

Using BDG Inequality, for $p\geq 1$, the definition of the quadratic variations given by (\ref{eqDefQuadVarW}, \ref{eqDefQuadVarW}) and Lemma \ref{lemma:massUniform}:
\begin{align*}
   & \EE\left[\left|\Fcal_{t,\varphi}(\mu^K) \right|^p \right]\\
    &\leq
    2^p\left( \EE\left[\sup_{s\in [0,t]}\left|M_s^{K,\phi,W} \right|^p \right]
    +\EE\left[\sup_{s\in [0,t]}\left|M_s^{K,\phi,\Ncal} \right|^p \right]
    \right)\\
   & \leq
    2^pC_p\left( \EE\left[\left\langle M^{K,\phi,W} \right\rangle^\frac{p}{2}_t \right]
    +\EE\left[\left\langle M^{K,\phi,\Ncal} \right\rangle^\frac{p}{2}_t \right]
    \right) \\
   & \leq
    \frac{2^pC_pt^\frac{p}{2}\left(\sup_{s\in[0,t]}\norm{\phi_s}_\infty^p+\sup_{s\in[0,t]}\norm{\dx \phi_s}_\infty^p\right)}{K^\frac{p}{2}} \EE\left[\sup_{s\in [0,t]}\langle \mu_s^K,1 \rangle^\frac{p}{2} \right]\\
   & \leq
    \frac{2^pC_p't^\frac{p}{2}\left(\sup_{s\in[0,t]}\norm{\phi_s}_\infty^p+\sup_{s\in[0,t]}\norm{\dx \phi_s}_\infty^p\right)}{K^\frac{p}{2}}  \left(\langle \mu_0^k, 1\rangle)^\frac{p}{2} + T^\frac{p}{2} +(T/K)^{\frac{p}{4}}\right)
\end{align*}
In conclusion:
\begin{align}
\label{eq:limEEFcalMuKisZero}
    \lim_{K\to +\infty} \EE\left[\left|\Fcal_{t,\varphi}(\mu^K) \right|^p \right] = 0
\end{align}

\textbf{Step b):} We prove that $|\Fcal(\mu^K)|^p\to |\Fcal(\mu^\infty)|^p$ in law.

Knowing that any subsequential limit $\mu^\infty$ of $(\mu^K)$ satisfies a.s. $\mu^\infty \in C([0,T], \Mcal_F(\R))$, we will show that $\Fcal_{t,\varphi}$ is (sequentially) continuous at points $\nu \in C([0,T], \Mcal_F(\R))$.

Let $(\nu^K)\in (\DD([0,T],\Mcal_F(\R))^\N$ such that $\lim_{K\to +\infty} \nu^K = \nu\in C([0,T], \Mcal_F(\R))$. The limit being continuous in time, we have that the convergence is uniform, \textit{i.e.} $\lim_{K\to +\infty} \\ \sup_{t\in [0,T]} d(\nu^K_t,\nu_t) =0$ (see Chapter 6 in \cite{pollard1984}), where $d$ is the Prokhorov metric (see Appendix A 2.5 in \cite{daley2003}). As a consequence, $\langle \nu_t^K,1\rangle$ is bounded uniformly in $K$ and in $t\in [0,T]$. We show the convergence for two terms and the other terms in $\Fcal_{t,\varphi}$ are dealt with in a similar manner. \\
By noticing that $\langle \tilde\Hcal \ast \nu^K_s, \dx^2\varphi_s\rangle \leq \sup_{s\in[0,T]}\norm{\dx^2\varphi_s}_1 \sup_{K\in \N, s\in[0,T]}\langle \nu_t^K,1\rangle$, the dominated convergence theorem can be applied: 
\begin{align*}
    &\lim_{K\to +\infty}\int_0^t \langle \tilde\Hcal \ast \nu^K_s, \dx^2\varphi_s \rangle ds = \int_0^t \lim_{K\to +\infty}\langle \tilde\Hcal \ast \nu^K_s, \dx^2\varphi_s \rangle ds =
    \int_0^t \lim_{K\to +\infty}\langle  \nu^K_s, \Hcal \ast \dx^2\varphi_s\rangle ds
\end{align*}
Next, as the convergence $\lim_{K\to +\infty} \nu_s^K = \nu_s$ is weak, $\Hcal \ast \dx^2\phi$ is an acceptable test function and:
\begin{align*}
    \int_0^t \lim_{K\to +\infty}\langle  \nu^K_s, \Hcal \ast \dx^2\varphi_s\rangle ds=
    \int_0^t \langle  \nu_s, \Hcal \ast \dx^2\varphi_s \rangle ds= 
    \int_0^t \langle \tilde \Hcal \ast \nu_s,  \dx^2\varphi_s\rangle ds
\end{align*}
Next, observing that $\tilde\Hcal \ast \nu_s^K(x) = \nu_s^K([x,+\infty))$ and converges pointwise to $\tilde\Hcal \ast \nu_s(x) $ for every point of continuity $x$ of $\tilde\Hcal \ast \nu_s $. $\tilde\Hcal \ast \nu_s $, being a nonincreasing function, can have at most countably many points of discontinuity. Thus, a.e. $\lim_{K\to +\infty}\tilde\Hcal \ast \nu_s^K(x) = \tilde\Hcal \ast \nu_s(x)$. Furthermore, noticing that $A(p)\leq p$, the dominated convergence theorem applies again: 
\begin{align*}
    &\lim_{K\to +\infty}\int_0^t \langle A(\tilde\Hcal \ast \nu^K_s), \dx\varphi_s \rangle ds = 
    \int_0^t \lim_{K\to +\infty}\langle A(\tilde\Hcal \ast \nu^K_s), \dx\varphi_s \rangle ds 
\end{align*}
Then, as $\dx \varphi_s$ is compactly supported, we have that:
\begin{align*}
   \left|  A\left(\tilde\Hcal \ast \nu^K_s(x)\right) \dx\varphi_s(x) \right| \leq \sup_{s\in[0,T]}\norm{\dx^2\varphi_s}_1 \sup_{K\in \N, s\in[0,T]}\langle \nu_t^K,1\rangle \mathbbm{1}_{\text{supp } \dx \varphi_s}(x),
\end{align*}
and we may apply once more the dominated convergence theorem. Finally, using the continuity of $A$, we obtain:
\begin{align*}
    &\int_0^t \lim_{K\to +\infty}\langle A(\tilde\Hcal \ast \nu^K_s), \dx\varphi_s \rangle ds \\
    =&\int_0^t \left\langle A\left(\lim_{K\to +\infty}\tilde\Hcal \ast \nu^K_s\right), \dx\varphi_s \right\rangle ds \\
    =& 
    \int_0^t \langle A(\tilde\Hcal \ast \nu_s), \dx\varphi_s \rangle ds 
\end{align*}
In conclusion, we have shown that:
\begin{align*}
    \lim_{K\to +\infty} \Fcal_{t,\varphi}\left(\nu^K\right)=\Fcal_{t,\varphi}(\nu),
\end{align*}
and thus that $\Fcal_{t,\varphi}$ is continuous at every point $\nu \in C([0,T],\Mcal_F(\R))$, which establishes that $|\Fcal(\mu^K)|^p\to |\Fcal(\mu^\infty)|^p$ in law.

\textbf{Step c):} We prove the convergence of the mean, \textit{i.e.} $\lim_{K\to +\infty}\EE\left[|\Fcal(\mu^K)|\right]=\EE\left[|\Fcal(\mu^\infty)|\right]$

The convergence in law combined with uniform integrability implies the convergence of the mean (see Proposition 2.3 in Appendixes in \cite{ethier1986}). Therefore, let us show that the sequence $\left(|\Fcal_{t,\phi}(|\mu^K)|\right)_{K}$ is uniformly integrable. By straightforward computations, we have that: 
\begin{align*}
    \left|\Fcal_{t,\phi}(\mu^K)\right|\leq C_{t,\phi}\left( 1+ \sup_{s\in [0,T]}\langle\mu_s^K,1\rangle\right)
\end{align*}
But according to Lemma \ref{lemma:massUniform} and the assumption that $\langle \mu_0^K,1\rangle$ has a second order moment, the $L^2$-norm of $\sup_{s\in [0,T]}\langle\mu_s^K,1\rangle$ is uniformly bounded. Thus $\left|\Fcal_{t,\phi}(\mu^K)\right|$ is uniformly integrable (see Proposition 2.2 in Appendixes in \cite{ethier1986}), which establishes that:
\begin{align*}
\EE\left[\left|\Fcal_{t,\varphi}(\mu^{\infty}) \right| \right]=0,
\end{align*}
This concludes the proof of Claim (\ref{eq:FcalMuInftyIsZero}) and thus the support of $\tilde\Hcal \ast \mu^\infty$ consists of  weak solutions to Equation (\ref{eq:limitPDE-F}).

\section{Regularity of and Uniqueness of the Limit}
\label{sec-uniqueness}

In this Section, we prove that the support of $\tilde\Hcal \ast \mu^\infty$ consists of a single point under certain assumptions on the initial datum specified in Theorem \ref{thm:uniqueness} and thus that the limit is unique. 

In the first Subsection, we prove the regularity result of Proposition \ref{prop:reg}, namely that a weak solution $F\in L^\infty([0,T]\times \R)$ to Equation (\ref{eq:weak-sol}) satisfies $F\in C([0,T],W^{1,\infty}(\R))$, that $u:=-\partial_x F \in C([0,T],L^p(\R))$ for all $p\in [1,\infty]$, and that $u$ is a weak solution to Equation (\ref{eq:weak-sol_u}) .

In the second Subsection, using the latter regularity result, we prove Theorem \ref{thm:uniqueness}, which states that the solution $F\in L^\infty([0,T]\times \R)$ to Equation (\ref{eq:weak-sol}) (or equivalently $u$ solution to Equation (\ref{eq:weak-sol_u})) is unique under additional assumptions on the initial datum $u_0$. 

Before moving on to the proof of these two results, let us notice that $F$ in the support of $\tilde\Hcal \ast \mu^\infty$ satisfies $F\in L^\infty([0,T]\times \R)$. Indeed, using Lemma \ref{lemma:massUniform} and the fact that $\mu_t\geq 0$, we have that
\begin{align*}
        \|F_t\|_\infty = F_t(-\infty) = \langle \mu_t^\infty,1\rangle \leq \langle \mu_0,1 \rangle +t.
    \end{align*}
Thus, the regularity result of Proposition \ref{prop:reg} applies to the limit obtained in Section \ref{sec-convergence}, and so does the uniqueness result of Theorem \ref{thm:uniqueness}, as long as the initial datum $u_0$ satisfies the assumptions of the latter theorem. Hence, we may claim that the limit obtained in Section \ref{sec-convergence} is unique.

\subsection{Regularity: Proof of Proposition \ref{prop:reg}.}
\label{sec-uniqueness-reg}

\begin{lemma}
\label{lem:rep-F}
    [A weak solution $F$ to Equation (\ref{eq:weak-sol}) is a mild solution] Let $F\in L^\infty([0,T]\times \R)$ be a weak solution of Equation (\ref{eq:limitPDE-F}), then $F$ satisfies for all $t\in[0,T]$ and almost all $x\in \R$, the following representation formula:
\begin{align}
\label{eq:RepForm-F}
F_t= e^{t\dx^2}F_0 + \int_0^t \left(-\chi \dx e^{(t-s)\dx^2}  A(F_s)  + e^{(t-s)\dx^2} B(F_s) \right)ds .
\end{align}
\end{lemma}

\begin{proof}
Since $F\in L^\infty([0,T]\times \R)$ and because of the density of $C^\infty_c([0,T]\times\R)$ in $L^\infty([0,T],W^{2,1}(\R))$, we can extend the weak formulation (\ref{eq:weak-sol}) to test functions in $L^\infty([0,T],W^{2,1}(\R))$.

Then, let $\varphi \in C^\infty_c(\R)$ and consider $\phi_s = e^{(t-s)\partial_x^2}\varphi$, which may be used as a test function by the preceding point and standard properties of the heat operator. By using the fact that $\partial_s \phi_s + \partial_x^2\phi_s=0$, that $\phi_t=\varphi$ and $\phi_0=e^{t\partial_x^2}\varphi$, we have that:
    \begin{align*}
       & \langle F_t, \varphi \rangle -\langle F_0, e^{t\partial_x^2}\varphi\rangle \\
        = & \int_0^t  \left\langle F_s, \partial_s \phi_s + \partial_x^2\phi_s\right\rangle ds +  \int_0^t \left( \left\langle  \chi  A(F_s),  \partial_x e^{(t-s)\partial_x^2}\varphi\right\rangle + \left\langle  B(F_s),  e^{(t-s)\partial_x^2}\varphi\right\rangle \right) ds\\
        = &  \int_0^t \left( \left\langle  \chi  A(F_s),  \partial_x e^{(t-s)\partial_x^2}\varphi\right\rangle + \left\langle  B(F_s),  e^{(t-s)\partial_x^2}\varphi\right\rangle \right) ds.
    \end{align*}
    Finally using the fact that the heat operator $e^{(t-s)\dx^2}$ and its derivative $\dx e^{(t-s)\dx^2}$ act on distribution as a convolution and by the fact that $A(F),B(F)\in L^\infty([0,T]\times \R)$ (since $F\in L^\infty([0,T]\times \R)$ and $A(\cdot),B(\cdot)$ are Lipschitz continuous), we can rewrite this equality as:
    \begin{align*}
        \left\langle F_t - e^{t\dx^2}F_0 - \int_0^t \left(-\chi \dx e^{(t-s)\dx^2}  A(F_s)  + e^{(t-s)\dx^2} B(F_s) \right)ds , \varphi \right\rangle =0,
    \end{align*}
    which concludes the proof as $\varphi \in C^\infty_c(\R)$ is arbitrary.
\end{proof}

Before proving Proposition \ref{prop:reg}, we state the following standard bounds on the heat operator, which are a consequence of the expression of the heat operator $e^{t\partial_x^2}$ and Young's inequality for convolutions:
    \begin{align}
       \label{bd:Heat} &\norm{e^{t\partial_x^2} f}_p \leq \norm{ f}_p,\\
       \label{bd:Heatdx} &\norm{\partial_x e^{t\partial_x^2} f}_p \leq \frac{C}{\sqrt{t}}\norm{ f}_p, \\
        &\norm{\partial_x^2 e^{t\partial_x^2} f}_p \leq \frac{C}{t}\norm{ f}_p 
        \label{bd:Heatdxx}.
\end{align}
Furthermore, we will need two more bounds involving the $C^{0,\alpha}(\R)$ space equipped with the norm $\|\cdot\|_{C^{0,\alpha}(\R)} = \|\cdot\|_\infty+[\cdot]_\alpha$. In order to establish these bounds, we use the real interpolation theory with the K-method (see Chapter 1 in \cite{lunardi2009}), according to which $C^{0,\alpha}(\R)=\left(C^{0}(\R),C^{1}(\R) \right)_{\alpha,\infty}$, where $C^{k}(\R)$ is the space of functions $k$-times differentiable with bounded derivatives, equipped with its usual $W^{k,\infty}(\R)$-norm. For the first bound, we use the fact that $\left\|\dx e^{t\partial_x^2}\right\|_{\Lcal(C^{0}(\R),C^{0}(\R))}\leq Ct^{-\frac{1}{2}}$ (from Bound (\ref{bd:Heatdx})) and $\left\|\dx e^{t\partial_x^2}\right\|_{\Lcal(C^{0}(\R),C^{1}(\R))}\leq C\left(t^{-\frac{1}{2}}+t^{-1}\right)$ (from Bounds (\ref{bd:Heatdx},\ref{bd:Heatdxx})), which then leads to 
\begin{align}
       \label{bd:HeatHolder}
       &\left\| \dx e^{t\partial_x^2} f\right\|_{C^{0,\alpha}(\R)} \leq C \left( \frac{1}{\sqrt{t}}+ \frac{1}{t^{\frac{1+\alpha}{2}}}\right)\norm{ f}_\infty
\end{align}
Similarly, as a consequence of the bounds $\left\|\dx^2 e^{t\partial_x^2}\right\|_{\Lcal(C^{0}(\R),C^{0}(\R))}\leq Ct^{-1}$ (from Bound (\ref{bd:Heatdxx})) and $\left\|\dx^2 e^{t\partial_x^2}\right\|_{\Lcal(C^{1}(\R),C^{0}(\R))}\leq Ct^{-\frac{1}{2}}$ (from Bound (\ref{bd:Heatdx})), we obtain
\begin{align}
\label{bd:HeatrevHolder}
       &\left\|\partial_x^2 e^{t\partial_x^2} f\right\|_\infty \leq \frac{C}{t^{1-\frac{\alpha}{2}}}\left\| f\right\|_{C^{0,\alpha}(\R)} 
    \end{align}

\begin{proof}[Proof of Proposition~\ref{prop:reg}]

Using the representation formula (\ref{eq:RepForm-F}), we show that each of the terms belongs to $C([0,T], W^{1,\infty}(\R))$.

First, since $u_0\in L^\infty(\R)\cap L^1(\R)$, we have that $F_0\in W^{1,\infty}(\R)$. Using the strong continuity of the heat operator on $W^{1,\infty}(\R)$, we have that
        \begin{align}
            e^{t\dx^2}F_0 \in C([0,T], W^{1,\infty}(\R)).
        \end{align}

Next, because $F\in L^\infty([0,T]\times \R)$, we have that $B(F) \in L^\infty([0,T]\times \R)$. Then, using Bounds (\ref{bd:Heat},\ref{bd:Heatdx}), we obtain that
        \begin{align*}
            &\left\| \int_0^t e^{(t-s)\dx^2} B(F_s)ds\right\|_\infty \leq C t  \left\|  B(F)\right\|_{\infty},\\
            &\left\| \partial_x\int_0^t e^{(t-s)\dx^2} B(F_s)ds\right\|_\infty \leq C \int_0^t \frac{ds}{\sqrt{t-s}}  \left\|  B(F)\right\|_{\infty} = C \sqrt{t}\left\|  B(F)\right\|_{\infty}.
            \end{align*}
        This establishes that $\int_0^t e^{(t-s)\dx^2} B(F_s)ds \in L^\infty([0,T],W^{1,\infty}(\R))$. \\
        Let us now show the continuity in time for all $t\in [0,T]$. We notice that
        \begin{align*}
           & \int_0^{t+h} e^{(t+h-s)\dx^2} B(F_s)ds - \int_0^t e^{(t-s)\dx^2} B(F_s)ds,\\
           =&\int_t^{t+h}e^{(t+h-s)\dx^2} B(F_s)ds +  \left(e^{h\dx^2}-I \right) \int_0^{t} e^{(t-s)\dx^2} B(F_s)ds .
        \end{align*}
        According to the above $\int_0^{t} e^{(t-s)\dx^2} B(F_s)ds \in W^{1,\infty}(\R)$ and by using again the strong continuity of the heat operator on $W^{1,\infty}(\R)$, we have that
        \begin{align*}
            \left(e^{h\dx^2}-I \right) \int_0^{t} e^{(t-s)\dx^2} B(F_s)ds  \overset{W^{1,\infty}(\R)}{\underset{h \to 0}{\longrightarrow}}0.
        \end{align*}
        Moreover, by the same type of bounds then above, we show that
        \begin{align*}
            \left\|\int_t^{t+h}e^{(t+h-s)\dx^2} B(F_s)ds \right\|_{W^{1,\infty}(\R)} \leq C (\sqrt{h}+h)\left\|  B(F)\right\|_{\infty}.
        \end{align*}
        Thus, $\int_0^{t} e^{(t-s)\dx^2} B(F_s)ds \in C([0,T],W^{1,\infty}(\R))$.

Finally, for the remaining term, the argument goes in two steps. We start by showing that $\int_0^t e^{(t-s)\dx^2} \partial_x(-\chi A(F_s))ds \in L^\infty([0,T],C^{0,\alpha}(\R))$ for any $\alpha \in (0,1)$, using Bound (\ref{bd:HeatHolder}):
        \begin{align*}
            &\left\|\int_0^t -\chi\partial_xe^{(t-s)\dx^2}  A(F_s)ds\right\|_{C^{0,\alpha}(\R)} \\
            \leq &C\int_0^t
            \left\| \partial_x e^{(t-s)\dx^2}  A(F_s)\right\|_{C^{0,\alpha}(\R)}  ds \\
            \leq&  C\left\|  A(F)\right\|_{\infty}\int_0^t \left( \frac{1}{\sqrt{t-s}}+ \frac{1}{(t-s)^{\frac{1+\alpha}{2}}}\right)ds\\
            \leq&  C\left\|  A(F)\right\|_{\infty}\int_0^t \frac{ds}{(t-s)^{\frac{1+\alpha}{2}}}, \text{ using }t\leq T <+\infty\\
            \leq&  C\left\|  A(F)\right\|_{\infty}t^{\frac{1-\alpha}{2}}.
        \end{align*}
        Hence $\int_0^t e^{(t-s)\dx^2} \partial_x(-\chi A(F_s))ds \in L^\infty([0,T],C^{0,\alpha}(\R))$ and by the representation formula (\ref{eq:RepForm-F}) and the above that $F \in L^\infty([0,T],C^{0,\alpha}(\R))$. Moreover since $A$ is Lipschitz continuous, we have that $A(F) \in L^\infty([0,T],C^{0,\alpha}(\R))$.
        Then using Bound (\ref{bd:HeatrevHolder}), we have that
        \begin{align*}
            &\left\|\partial_x\int_0^t-\chi \partial_x e^{(t-s)\dx^2}  A(F_s)ds\right\|_\infty\\
            \leq & C\int_0^t \left\| \partial_x^2e^{(t-s)\dx^2} A(F_s)ds\right\|_\infty ds \\
            \leq & C\left\|  A(F)\right\|_{L^\infty([0,T], C^{0,\alpha}(\R))}  \int_0^t \frac{ds}{(t-s)^{1-\frac{\alpha}{2}}}\\
            \leq & C\left\|  A(F)\right\|_{L^\infty([0,T], C^{0,\alpha}(\R))} t^{\frac{\alpha}{2}} .
        \end{align*}
        Thus, $\int_0^t -\chi \partial_xe^{(t-s)\dx^2} A(F_s)ds \in L^\infty([0,T],W^{1,\infty}(\R))$. The continuity in time is established along the same lines than above.

Thus, combining the result on the three terms above, we have that
        \begin{align*}
            F \in C([0,T], W^{1,\infty}(\R)).
        \end{align*}
\begin{lemma}
\label{lem:rep-u}
    Consider $F$ satisfying the assumptions of Proposition \ref{prop:reg} and set $u:=-\partial_x F$. Then $u$ is a weak solution of Equation (\ref{eq:weak-sol_u}) and it satisfies for all $t\in[0,T]$ and almost all $x\in \R$ the following representation formula:
    \begin{align}
\label{eq:RepForm-u}
u_t= e^{t\dx^2}u_0 + \int_0^t e^{(t-s)\dx^2}\left( \dx  \left(\chi \tilde a(F_s)u_s \right) + \tilde b(F_s)u_s \right)ds.
\end{align}
\end{lemma}
\begin{proof}[Proof of Lemma~\ref{lem:rep-u}]
The representation formula (\ref{eq:RepForm-u}) for $u$ is an immediate consequence of the representation formula (\ref{eq:RepForm-F}) for $F_t$, where we use the chain rule $\partial_xA(F_t)=\tilde{a}(F_t)\partial_x F_t=-\tilde{a}(F_t)u_t$, which may be applied as $F_t \in W^{1,\infty}(\R)$, according to the above. 
\end{proof}
Now, let us show the remaining regularity result on $u$.

As a mere restatement of the fact that $F\in C([0,T],W^{1,\infty}(\R))$, we have that $u\in C([0,T], L^\infty(\R))$.

    Finally, since $F\in L^\infty([0,T]\times\R)$ and the fact that $u\geq 0$, we have that
    \begin{align}
        \|u_t\|_1 = \|F_t\|_\infty,
    \end{align}
    and thus that $u\in L^\infty([0,T],L^1(\R))$. The continuity of the mapping $t\mapsto u_t$ is then a consequence of the representation formula (\ref{eq:RepForm-u}) and similar arguments than the ones for the continuity in time of $F$. By interpolation, we may hence conclude that for all $p\in[1,\infty]$,
    \begin{align}
        u\in C([0,T], L^p(\R)).
    \end{align}
\end{proof}

\subsection{Uniqueness: Proof of Theorem \ref{thm:uniqueness}.}
\label{sec-uniqueness-unique}

Given the regularity result of Proposition \ref{prop:reg}, we may now move to the proof of Theorem~\ref{thm:uniqueness}.

\begin{proof}[Proof of Theorem~\ref{thm:uniqueness}]
    Suppose that there exists two solutions $F_{1},F_{2}$ satisfying Equation \ref{eq:limitPDE-F}. Then by the above, we may consider $u_{1}=-\partial_x F_{1},u_{2}=-\partial_x F_{2,t}$, which both satisfy the representation formula (\ref{eq:RepForm-u}). Let us denote $\bar{x}_1(t), \bar{x}_2(t)$ the thresholds, \textit{i.e.} such that $F_{1,t}(\bar{x}_1(t)), F_{2,t}(\bar{x}_2(t))=1$.

    \textbf{Step a):} There exists $T>0$, such that for all $t\in [0,T], \bar{x}_1(t), \bar{x}_2(t) \in [\bar{x}(0)-\eta, \bar{x}(0)+ \eta]$.
    
    The linear functional $f \in L^1(\R)\mapsto \int_{\bar{x}(0)-\eta}f(x)dx$ (resp. $f \in L^1(\R)\mapsto \int_{\bar{x}(0)+\eta}f(x)dx$) is continuous. Furthermore, according to Proposition \ref{prop:reg}, $u_{i}\in C([0,T], L^1(\R))$, thus the maps $t\mapsto F_{i,t}(\bar{x}(0)-\eta)$ (resp. $t\mapsto F_{i,t}(\bar{x}(0)-\eta)$) are continuous. But, since $F_{i,0}(\bar{x}(0)-\eta)>1$ (resp. $F_{i,0}(\bar{x}(0)+\eta)<1$), there exists $T>0$ such that for all $t\in [0,T]$, $F_{i,t}(\bar{x}(0)-\eta)>1$ (resp. $F_{i,t}(\bar{x}(0)+\eta)<1$). Thus for all $t\in [0,T]$, $\bar{x}_i(t) \in [\bar{x}(0)-\eta, \bar{x}(0)+ \eta]$

    \textbf{Step b): }There exists $T>0$, such that for all $t\in [0,T], \inf_{x \in [\bar{x}(0)-\eta, \bar{x}(0)+ \eta]} u_{i,t}(x)\geq \frac{\underline{u}}{2}$.

    Using Proposition \ref{prop:reg}, we have that $u_{i}\in C([0,T], L^\infty(\R))$. Furthermore, by assumption $\inf_{x \in [\bar{x}(0)-\eta, \bar{x}(0)+ \eta]}u_0(x)\geq \underline{u}$. Hence there exists $T>0$ such that for all $t\in [0,T]$, $ \inf_{x \in [\bar{x}(0)-\eta, \bar{x}(0)+ \eta]} u_{i,t}(x)\geq \frac{\underline{u}}{2}$.

    \textbf{Step c): } For all $t\in [0,T]$,  $\frac{\underline{u}}{2}|x_1(t)-x_2(t)| \leq  \norm{u_{2,t}-u_{1,t}}_1$.

    Suppose that for a fixed $t\in [0,T], \bar{x}_1(t)\leq \bar{x}_2(t)$, then
    $$\int_{\bar{x}_1(t)}^{\bar{x}_2(t)}u_{1,t}(x)dx\geq \inf_{x \in [\bar{x}(0)-\eta, \bar{x}(0)+ \eta]} u_{i,t}(x)|x_1(t)-x_2(t)|\geq \frac{\underline{u}}{2}|x_1(t)-x_2(t)|, 
    $$
    where we have used the fact that $\bar{x}_i(t)\in [\bar{x}(0)-\eta, \bar{x}(0)+ \eta]$, according to Step a). On the other hand, we have that
    \begin{align*}
        &\int_{\bar{x}_1(t)}^{\bar{x}_2(t)}u_{1,t}(x)dx\\
        =& \int_{\bar{x}_1(t)}^{+\infty}u_{1,t}(x)dx-\int_{\bar{x}_2(t)}^{+\infty}u_{2,t}(x)dx+\int_{\bar{x}_2(t)}^{+\infty}u_{2,t}(x)dx\int_{\bar{x}_2(t)}^{+\infty}u_{1,t}(x)dx\\
        =& 1- 1 +\int_{\bar{x}_2(t)}^{+\infty}(u_{2,t}(x)-u_{1,t}(x))dx\\
        \leq& \norm{u_{2,t}-u_{1,t}}_1.
    \end{align*}
    Using the same argument when $\bar{x}_1(t)>\bar{x}_2(t)$, we obtain that
    \begin{align}
    \label{bd:thresholdL1}
        \frac{\underline{u}}{2}|x_1(t)-x_2(t)| \leq  \norm{u_{2,t}-u_{1,t}}_1.
    \end{align}

    \textbf{Step d): } $\sup_{t\in [0,T]}\norm{u_{2,t}-u_{1,t}}_1 \leq C\sqrt{T}\sup_{t\in [0,T]}|x_1(t)-x_2(t)| .$\\
    
    Using the representation formula (\ref{eq:RepForm-u}) and the estimates (\ref{bd:Heat},\ref{bd:Heatdx}), we have that
    \begin{align}
        &\norm{u_{2,t}-u_{1,t}}_1\\
        \leq & C\int_0^t\left( \frac{1}{\sqrt{t}}+1\right)\left( \norm{\tilde a(F_{1,s})u_{1,s}-\tilde a(F_{2,s})u_{2,s}}_1 + \norm{\tilde b(F_{1,s})u_{1,s}-\tilde b(F_{2,s})u_{2,s}}_1\right)ds
    \end{align}
    Furthermore,
    \begin{align*}
        &\norm{\tilde a(F_{1,s})u_{1,s}-\tilde a(F_{2,s})u_{2,s}}_1 \\ \leq &\norm{\tilde a(F_{1,s})-\tilde a(F_{2,s})}_1\norm{u_{1,s}}_\infty + \norm{\tilde a(F_{2,s})}_\infty \norm{u_{1,s}-u_{2,s}}_1 \\  
        \leq  & |x_1(s)-x_2(s)|\norm{u_{1,s}}_\infty + \norm{u_{1,s}-u_{2,s}}_1.
    \end{align*}
    Similarly,
        \begin{align*}
        &\norm{\tilde b(F_{1,s})u_{1,s}-\tilde b(F_{2,s})u_{2,s}}_1 
        \leq  |x_1(s)-x_2(s)|\norm{u_{1,s}}_\infty + \norm{u_{1,s}-u_{2,s}}_1.
    \end{align*}
    Hence,
    \begin{align*}
        &\norm{u_{2,t}-u_{1,t}}_1\\
        \leq &  C\int_0^t\left( \frac{1}{\sqrt{t}}+1\right)|x_1(s)-x_2(s)|\norm{u_{1,s}}_\infty + \norm{u_{1,s}-u_{2,s}}_1 ds \\
        \leq & C\left(\sup_{t\in [0,T]}\norm{u_{1,t}}_\infty \sup_{t\in [0,T]}|x_1(t)-x_2(t)| + \sup_{t\in [0,T]}\norm{u_{2,t}-u_{1,t}}_{1}\right) (\sqrt{t}+t).
    \end{align*}
    Taking the supremum over $t\in  [0,T]$ and rearranging the terms, we find that up to taking $T>0$ small enough:
    \begin{align*}
        &\sup_{t\in [0,T]}\norm{u_{2,t}-u_{1,t}}_{1} \\
        \leq& \frac{C\sup_{t\in [0,T]}\norm{u_{1,t}}_\infty(\sqrt{T}+T)}{1-C\sup_{t\in [0,T]}\norm{u_{1,t}}_\infty(\sqrt{T}+T)}\sup_{t\in [0,T]}|x_1(t)-x_2(t)|\\
        \leq& C'\sqrt{T}\sup_{t\in [0,T]}|x_1(t)-x_2(t)|.
    \end{align*}

    \textbf{Step e): } Uniqueness on an interval $[0,T]$.
    
    Combining this latter bound with bound (\ref{bd:thresholdL1}), we find that
    \begin{align*}
        \sup_{t\in [0,T]}|x_1(t)-x_2(t)|\leq \frac{2C'\sqrt{T} }{\underline{u}}\sup_{t\in [0,T]}|x_1(t)-x_2(t)|.
    \end{align*}
    Thus for $T>0$ small enough, we obtain that
    \begin{align*}
        \sup_{t\in [0,T]}|x_1(t)-x_2(t)| = 0.
    \end{align*}
    This establishes uniqueness of the thresholds $x_i(t)$ on the interval $[0,T]$ and thus also equality of the functions $u_{1,t}\equiv u_{2,t}$ on the interval $[0,T]$. 

    \textbf{Step f):} Uniqueness remains true for all time. 

    Finally by a maximum principle argument, we know that $u_t(x)>0$ for all $(t,x)\in(0,+\infty)\times \R$ and thus the assumptions of the Theorem at initial time $t=0$ are also satisfied for later times $t>0$. Thus the preceding argument can be repeated for all time, which establishes uniqueness of the solution.
\end{proof}

\section{Speed of Propagation and Ancestral Lineages: Numerical investigations}
\label{sec-numerics}

In this section, we present numerical investigations of the model. First, we start by recalling propagation properties of the deterministic model that are studied in \cite{demircigilhenderson}. Then, we investigate numerically propagation properties of the finite population stochastic model and compare them to the deterministic model.
Finally, we illustrate numerically a dichotomy between pulled and pushed waves through the lens of the ancestral lineage methodology that has recently been proposed independently in \cite{calvez2022,forien2022}.

The code used for the present numerical simulations as well as the generated data is publicly accessible \cite{PLMLAB}. 

\subsection{A Brief Overview of Properties of the Deterministic Model}

The deterministic model  has been investigated in \cite{demircigilhenderson}, as well as in \cite{demircigil2022} for an analogous model, and we recall some results.  

\begin{theorem}
\label{thm-tw}
There exists a minimal speed $\s^*$, such that there  exists a bounded and nonnegative traveling wave profile $u^\s(z)$, \textit{i.e.} $u_t(x)=u^\s(x-\s t)$ is a solution to Equation (\ref{eq:limitPDE}), if and only if $\s\geq \s^*$. Given $\s\geq \s^*$, the traveling wave profile $u^\s$ is unique up to translation. Moreover, the exact value of $\s^*$ is given by:
\begin{align}
\label{sigma-formula}
\s^* = \left\{\begin{array}{ll}
\chi+\frac{1}{\chi} & \text{ if }\chi>1 \\
2 & \text{ if } \chi\leq 1
\end{array}
\right. .
\end{align}
For $\chi>1$:
\begin{align}
    u^{\s^*}(z) = \left\{
\begin{array}{cc}
    \chi e^{-\chi z} & z>0 \\
    \chi & z\leq 0
\end{array}
    \right.
\end{align}
For $\chi\leq1$:
\begin{align}
    u^{\s^*}(z) = \frac{1}{2-\chi}\cdot  \left\{
\begin{array}{cc}
    ((1-\chi)z+1) e^{- z} & z>0 \\
    1 & z\leq 0
\end{array}
    \right.
\end{align}
\end{theorem}  

\begin{theorem}
\label{thm-propagation}
Let $u_t$ be a solution of Equation (\ref{eq:limitPDE}) with initial condition rapidly decreasing at $x=+\infty$, \textit{i.e.} there exists $\alpha>\max\{\chi,1\}$ and $C>0$, such that
\begin{align*}
    u_0(x)\leq C e^{-\alpha x}.
\end{align*}
Let $\bar{x}(t)$ be the position of the threshold defined by
$$ \int_{\bar{x}(t)}^{+\infty} u_t(x)dx = 1,
$$
\textit{i.e.} the large population equivalent of the position of the $K$-th particle. The following  behavior holds for $\bar{x}(t)$:
For $\chi>1$,
\begin{align}
    \label{propagation-chilarge}
    \bar{x}(t) = \sigma^* t+ \mathcal{O}(1).
\end{align}
For $\chi<1$,
\begin{align}
    \label{propagation-chismall}
    \bar{x}(t) = \sigma^* t-\frac{3}{2}\ln(t)+ \mathcal{O}(1).
\end{align}
For $\chi=1$,
\begin{align}
    \label{propagation-chi1}
    \bar{x}(t) = \sigma^* t-\frac{1}{2}\ln(t)+ \mathcal{O}(1).
\end{align}
\end{theorem}   

The logarithmic correction term in (\ref{propagation-chismall},\ref{propagation-chi1}) has first been established for the Fisher/ Kolmogorov-Petrovsky-Piskunov model \cite{bramson1978} and is often referred to as the Bramson shift.

The dichotomy on the existence of the Bramson shift (depending on whether $\chi>1$ or $\chi\leq 1$) is tied to a dichotomy of traveling waves known as \textit{pushed} or \textit{pulled} waves, introduced in \cite{stokes1976}. 
Following this latter work a pulled wave is one where $\s^*=2$ and the dynamics of the wave is dictated by the linearization of the model at $x=+\infty$. A pushed wave is one where $\s^*>2$, \textit{i.e.} the wave is traveling faster than the dynamics of the linearizaton at $x=+\infty$. Qualitatively, in the pulled case, the wave is solely driven by growth and diffusion at the leading edge of the front with negligible contribution from the overall population, whilst in the pushed case, the wave is subject to a significant contribution from the overall population to the
net propagation. 

The study of pulled and pushed waves in \cite{demircigil2022}, following the methodology of neutral fractions proposed in \cite{roques2012,garnier2012}, applies exactly to the present model. A neutral fraction $\nu_t $ describes the contribution of a subpart of the traveling wave. It corresponds to labeling particles in the traveling wave with a label that does not interfere with the dynamics and that is conserved by both particles when the particle branches.  Mathematically for the considered model, a neutral fraction $\nu_t$ corresponds to a solution, in the moving frame prescribed by the traveling wave solution, \textit{i.e.} $(t,z)=(t,x-\s t)$, to the parabolic equation:
\begin{align}
    \label{eq:neutralfraction}
    &\dt \nu_t + L \nu_t =0,\text{ where }\\
    &L:=-\dzz -\beta(z)\dz\text{ and }\beta(z)=\sigma-\chi \mathbbm{1}_{z<0}+2\frac{\dz u^{\s}}{u^{\s}}.
    \label{def:beta}
\end{align}
We can observe that $\nu_t$, the solution to Equation (\ref{eq:neutralfraction}) with initial datum $\nu_0\equiv 1$, which corresponds initially to the whole of the traveling wave, remains constant $\nu_t \equiv 1$ (as $L1\equiv 0$), which is consistent with the fact that the whole of the traveling wave maintains its own propagation. 

In the pulled case $\chi\leq 1$, any neutral fraction $\nu(t)$ that initially is not constituted of a substantial part of the leading edge, for instance such that $\nu_0(z)z^2$ is integrable at $z=+\infty$ (excluding thereby the trivial case $\nu_0\equiv 1)$, will converge to 0 and go extinct. Hence only particles (and their offspring) that are at the very leading edge of the traveling wave may not go extinct in the traveling wave. In the pushed case $\chi>1$, any neutral fraction $\nu(t)$ converges exponentially to a constant in an appropriate $L^2$-setting. Hence any particle (and not just the ones at the leading edge) will remain part of the wave and contribute to its propagation.

Of note, the limit case $\chi=1$ is a pulled wave according to the above, but has some specific properties such as a distinct factor for the Bramson shift (\ref{propagation-chi1}) and a further difference will be illustrated below.
Therefore following \cite{an2021}, we will refer to this case as a \textit{pushmi-pullyu} wave. 

\subsection{Numerical Investigations of the Propagation Speed}

In this subsection, we provide numerical evidence that corroborates Conjecture \ref{conjecture:TW}, whose claims we recall: 
\begin{enumerate}
    \item Linear spreading in the finite population regime. There exists a constant propagation speed $\s^K$ such that in an appropriate convergence sense, $\lim_{t\to+\infty} \frac{\xi^K(t)}{t}=\s^K$, where $\xi^K_t=H_K(\mu^K_t)$ is the position of the $K$-th particle at time $t$.
    \item Local convergence to a traveling wave profile in the moving frame. Let $\Kcal\subset \R$ be a compact subset, then $\hat{\mu}^{K}_t:=\mu^K_t(\cdot -\xi^K(t))$ converges in law to a stationary distribution $\hat{\mu}^{K}_{\infty}$ on $\Kcal$.
    \item Convergence of the finite population propagation speed to the deterministic propagation speed.  $\lim_{K\to+\infty} \s^K = \s^*$, where $\sigma^*=\left\{\begin{array}{ll}
    \chi+\frac{1}{\chi} & \text{ if }\chi>1 \\
    2 & \text{ if }\chi\leq 1
\end{array} \right.$.
\end{enumerate}

The algorithm for the numerical simulations follows the rules of the individual based model as described in Section \ref{sec-individualbasedmodel}. Although simulations have shown that the initial condition does not seem to play a role in the long time behavior of the system, we have used the following initial condition for our simulations, which mimics the shape of the traveling waves given by Theorem \ref{thm-tw} restricted to the interval $\left[-\frac{1}{\chi},+\infty\right)$:
\begin{align*}
    (\text{ID}) : \left\{
\begin{array}{l}
   N^K_0 = 2K.    \\
   \text{The positions of the first $K$ particles sample the exponential distribution } \mathcal{E}({\chi}). \\
   \text{The remaining $K$ particles sample the uniform distribution } \mathcal{U}\left(\left[ - \frac{1}{\chi},0\right]\right). 
\end{array}
    \right.
\end{align*}

In the long time behavior the distribution $\mu^K_t$ is reminiscent of the traveling wave profile $u^{\s^*}$ at least in a compact set around the position of the $K$-th particle (see Figure \ref{fig:histogram}), which is consistent with Item 2 of Conjecture \ref{conjecture:TW}. 

Next, we observe that the spreading of the $K$-th particle evolves linearly (see Figure 
\ref{fig:speed:sfig1} for the case $\chi =\frac{1}{2},2$ and $K=1,4096$), which supports Item 1 of Conjecture \ref{conjecture:TW}. To illustrate this, we recall that $\xi^K_t=H_K(\mu^K_t)$ is the position of the $K$-th particle and we plot $\frac{\xi^K_t-\xi^K_{100}}{t-100}$, \textit{i.e.} the spreading speed of the $K$-th particle after $t\geq 100$, which is done in order to avoid any stark effect of the initial condition. This quantity seems to converge to a finite value, indicating a linear spreading. Then, we investigate the evolution of the spreading on the time interval $[100,200]$ as $K$ gets larger (see Figure \ref{fig:speed:sfig2}). For the case $\chi=2$, the spreading speed as $K$ increases almost perfectly matches the wave speed $\s^*$ of the deterministic model given by Formula (\ref{sigma-formula}), corroborating Item 3 of Conjecture \ref{conjecture:TW}. For the case $\chi=\frac{1}{2}$ the spreading speed is slightly slower than $\s^*=2$. Yet, this seems to be consistent with the asymptotic behavior (\ref{propagation-chismall}) of the deterministic model, as the spreading speed over a duration of $100$ is approximately $\frac{\bar{x}(100)}{100}\approx 2-\frac{3}{2}\frac{\ln(100)}{100}\approx 1.93$.

\begin{figure}[t]
\begin{center}
\includegraphics[width=.7\linewidth]{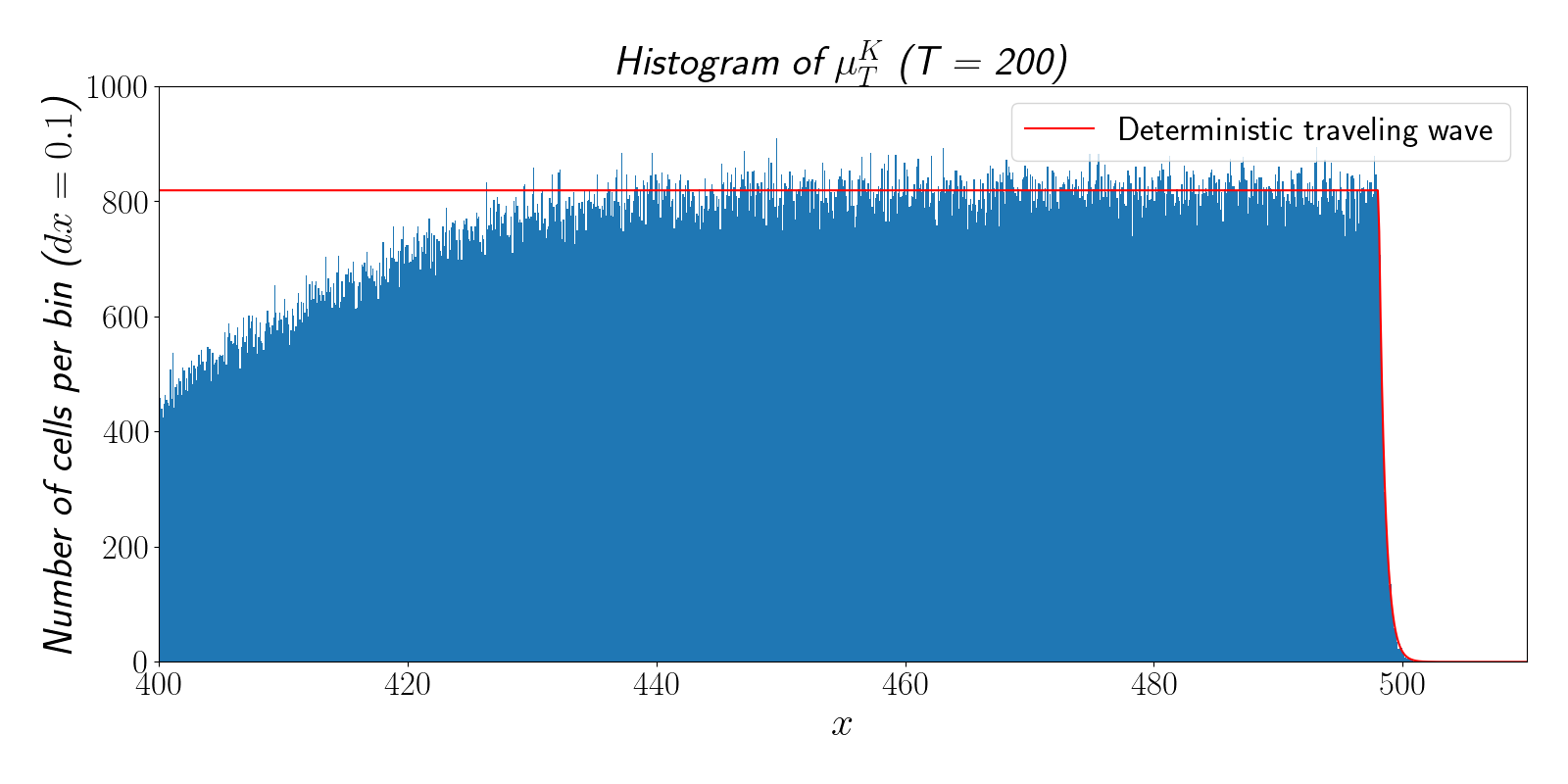}
\caption{Histogram of $\mu_T^K$ for $\chi=2,K=4096,T=200$ and initial data satisfying $(\text{ID})$. The width of a single bin in the histogram is $dx=0.1$. The red curve represents $y=C\left\{ \begin{array}{ll}
e^{-\chi (x-\xi^K_T))}& \text{ if } x>\xi^K_T\\
1 & \text{ if } x\leq \xi^K_T
\end{array} \right.$, where we recall that $\xi^K_T=H_K(\mu_T^K) $ denotes the position of the $K$-th particle at time $T$ and $C$ has been chosen \textit{a priori} with $C:=K\chi dx$, which is consistent with the discretization of the histogram.}
\label{fig:histogram}
\end{center}
\end{figure}

\begin{figure}[b]
\begin{center}
\begin{subfigure}{.5\linewidth}
  \centering
\includegraphics[width=\linewidth]{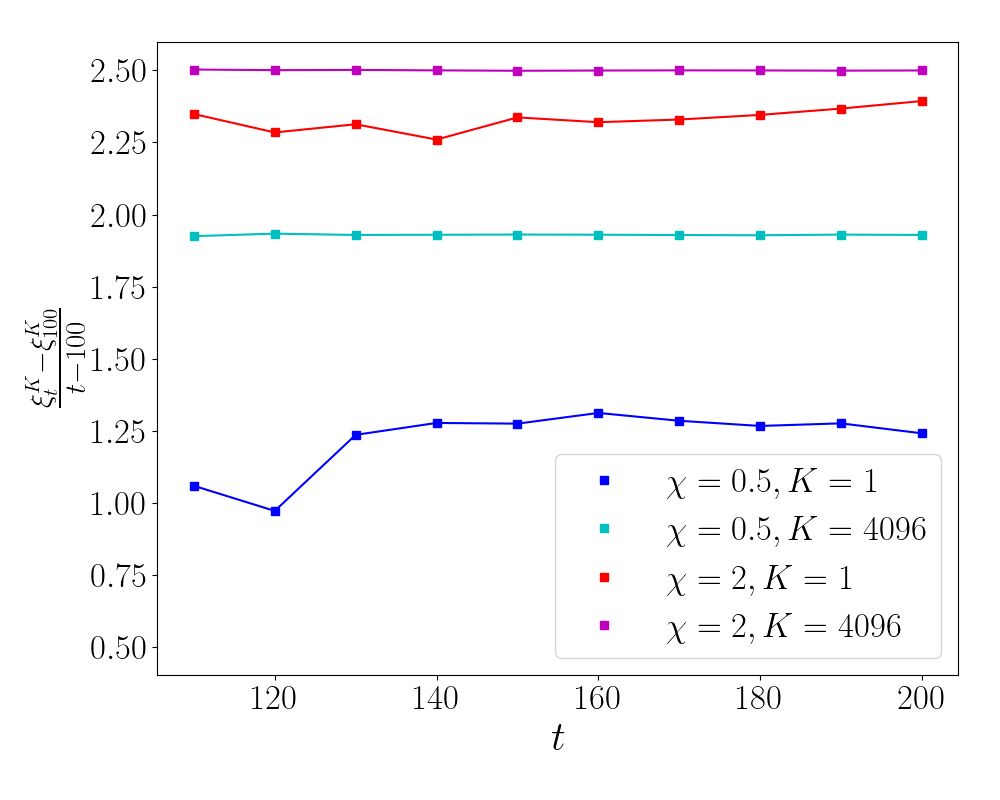}
  \caption{}
  \label{fig:speed:sfig1}
\end{subfigure}%
\begin{subfigure}{.5\linewidth}
  \centering
\includegraphics[width=\linewidth]{figure-speeds.png}
  \caption{}
  \label{fig:speed:sfig2}
\end{subfigure}
\caption{{(a)}: Graphic representation of $\frac{\xi^K_t-\xi^K_{100}}{t-100}$, where $\xi^K_t=H_K(\mu^K_t)$ is the position of the $K$-th particle at time $t$, for $K=4096$, $\chi=2$ (red), $\chi=\frac{1}{2}$ (blue) and initial data satisfying $(\text{ID})$. {(b)}:
Box plot of the spreading speed $\frac{\xi^K_{200}-\xi^K_{100}}{100}$ on the time interval $[100,200]$ with a sample size of $n=20$ for $\chi=2$, $\chi=\frac{1}{2}$, different values of $K$ and initial data satisfying $(\text{ID})$. The red (resp. blue) curve represents the sample mean $\frac{\overline{\xi^K_{200}-\xi^K_{100}}}{100}$ for $\chi=2$ (resp. $\chi=\frac{1}{2}$). The magenta (resp. cyan) curve represents the traveling wave speed $\s^*=\chi+\frac{1}{\chi}$ with $\chi=2$ (resp. $\s^*=2$ with $\chi=\frac{1}{2}$) of the deterministic model.    }
\label{fig:speed}
\end{center}
\end{figure}

\subsection{The Ancestral Lineage Methodology: An Alternative Viewpoint on Pushed and Pulled Waves}

\begin{figure}[t]
\begin{center}
\begin{subfigure}{.5\linewidth}
  \centering
\includegraphics[width=\linewidth]{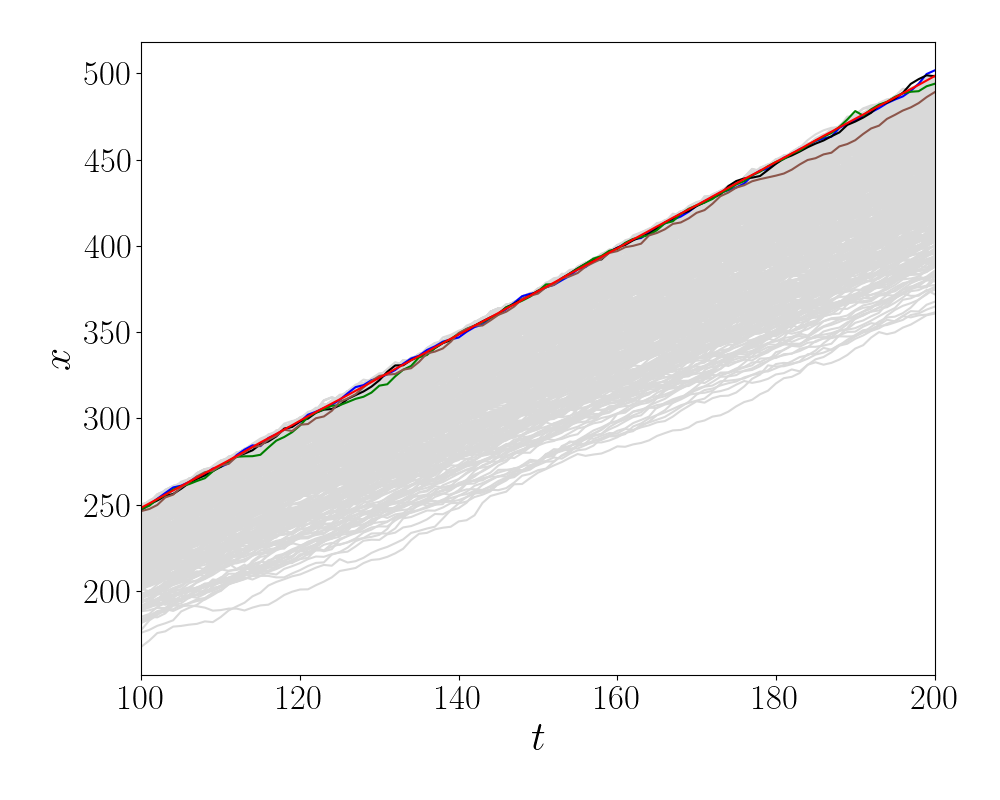}
  \caption{}
  \label{fig:trajectories:sfig1}
\end{subfigure}%
\begin{subfigure}{.5\linewidth}
  \centering
\includegraphics[width=\linewidth]{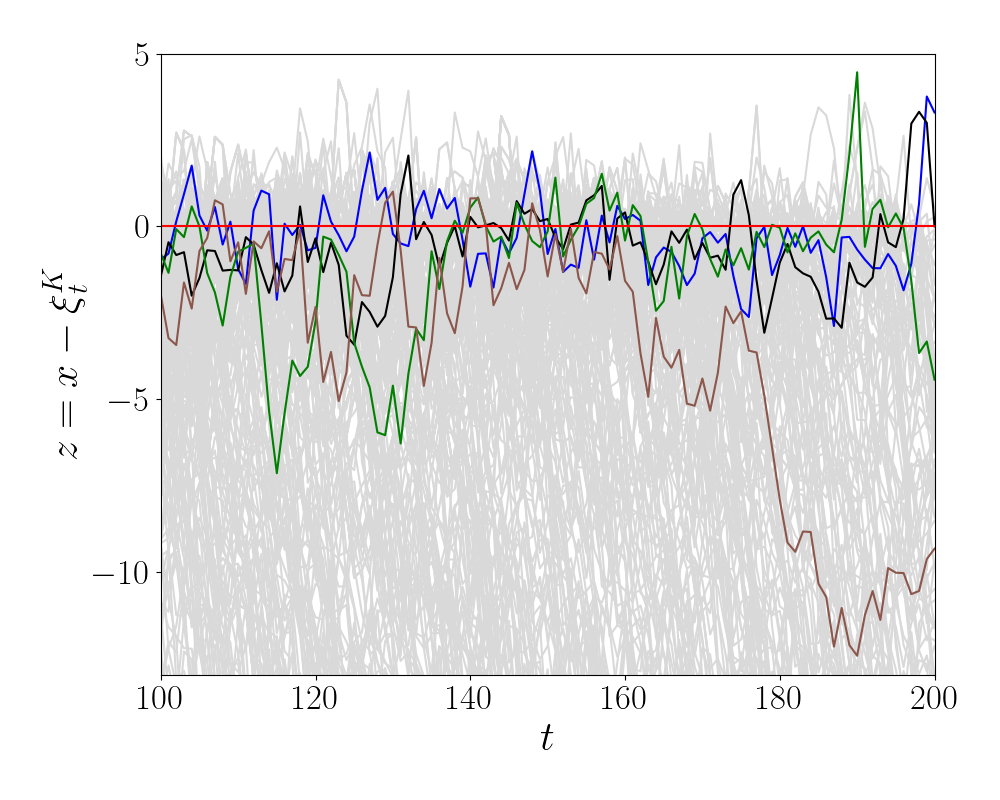}
  \caption{}
  \label{fig:trajectories:sfig2}
\end{subfigure}
\caption{Graphical representation of the particle trajectories for $t \in [100, 200] $, with
$\chi = 2, K=256$ and initial data satisfying $(\text{ID})$. The grey curves represent particle trajectories,
which have been sampled among the population (with $n=516$ among a total population $N^K_T=51633$ at time $T=200$). The red curve represents the position of the $K$-th particle and 4 particle trajectories have been selected arbitrarily to be highlighted for illustration purposes (blue curve corresponds to the particle with rank 1 at time $T=200$, black to the particle with rank $K$, green to the particle with rank $10K$, and brown to the particle with rank $20K$).
Figure (a) represents the
trajectories in the frame $(t, x)$, whereas Figure (b) represents the same trajectories
in the frame $(t,z)=(t,x-\xi^K_t)$.}
\label{fig:trajectories}
\end{center}
\end{figure}

In this Subsection, we apply the methodology of ancestral lineages, proposed independently in \cite{calvez2022,forien2022}, which can be seen as a backward in time analog of the neutral fraction framework described above. 

Roughly speaking, given a particle $i\in V^K_t$ at time $t$, we are interested in the position of its ancestor $Y_{s,t}^{i,K}$ at time $t-s$, for $s\geq0$. $\hat Y_{s,t}^{i,K}$ denotes the position of the ancestral lineage in the frame, which centers at $z=0$ the $K$-th particle, \textit{i.e.} $\hat{Y}_{s,t}^{i,K}=Y_{s,t}^{i,K}-\xi^K_s$. In Figure \ref{fig:trajectories}, such ancestral lineages have been represented in the stationary frame $(t,x)$ and in the moving frame $(t,z)=\left(t,x-\xi^K_t\right)$, centering the $K$-th particle. 

 Suppose that for $K<+\infty$ the distribution of the initial condition $\mu^K_0$ is given by the traveling wave $u^{\s^*}$. Then, in the moving frame $(t,z)=(t,x-\s^* t)$,  the distribution of $\mu^\infty_t$, the large population limit as $K\to +\infty$, is stationary, as $u^{\s^*}$ is a traveling wave solution of Equation (\ref{eq:limitPDE}), \textit{i.e.} a stationary solution in the moving frame $(t,z)=(t,x-\s^* t)$. 
Conjecture \ref{conjecture:AL}, which states analogous results to the results in \cite{calvez2022,forien2022} for the present model, then states that the large population limit $K\to +\infty$ of the ancestral lineage $\hat Y_{s,t}^{K}$ is a solution to the following SDE:
\begin{align}
\label{ancestSDE}
d\hat Y_{s,t} = \beta(\hat Y_{s,t})ds + \sqrt{2}dW_s.
\end{align}
As a consequence, $v(s)$ the probability distribution of $\hat{Y}_{s,t}$ satisfies the PDE:
\begin{align}
\label{ancestPDE}
\partial_s v+ \mathcal{L}v=0,
\end{align}
where $\mathcal{L}:=-\dzz+\dz\left(\beta(z) \cdot\right)$ and $\beta(z)$ is given by (\ref{def:beta}). In other words, the ancestral lineage distribution in the large population limit $K\to +\infty$ can be seen as a dual evolution problem of the neutral fraction evolution (\ref{eq:neutralfraction}). Let us also observe that Equation (\ref{ancestPDE}) is conservative, which is consistent with the fact that each particle has only a single ancestor at a given time.

The big mathematical challenge in applying rigorously the methods of \cite{calvez2022,forien2022} consists in the fact that the traveling wave solution $u^{\s^*}$ is not integrable, unlike in the aforementioned studies. Thus, the large population limit result Theorem \ref{thm:mainresult} established in Section \ref{sec-convergence} needs to be extended to initial distributions $\mu^K_0$ that are not of finite mass. In order to circumvent this issue in the numerical simulations, we conflate the stationary distribution $u^{\s^*}$ with a distribution $\mu^K_t$ that has evolved for sufficiently large $t$ (here $t=100$), so that in the moving frame at least close to the origin $z=0$ the distribution of $\mu^K_t$ is very close to the stationary distribution $u^{\s^*}$. Furthermore, we conflate the moving frame of the traveling wave $(t,z)=(t,x-\s^* t)$ with the moving frame centering the $K$-th particle $(t,z)=(t,x-\xi^K_t)$.

Neglecting these mathematical difficulties and using these approximations, we can then illustrate the dichotomy of pushed and pulled waves via the qualitative behavior of the ancestral lineage distribution given by Equation (\ref{ancestPDE}): The ancestral lineage distribution of a pushed wave converges to an equilibrium, whilst the ancestral lineage distribution of a pulled wave does not and its mean drifts to $z=+\infty$. Qualitatively in the pulled case, the position of the ancestors will in average drift indefinitely towards the leading edge of the traveling wave as $s$ increases, whilst in the pushed case, the distribution of the position of the ancestors converges to an equilibrium in the bulk of the wave.

\begin{figure}[t]
\begin{center}
\begin{subfigure}{.5\linewidth}
  \centering
\includegraphics[width=\linewidth]{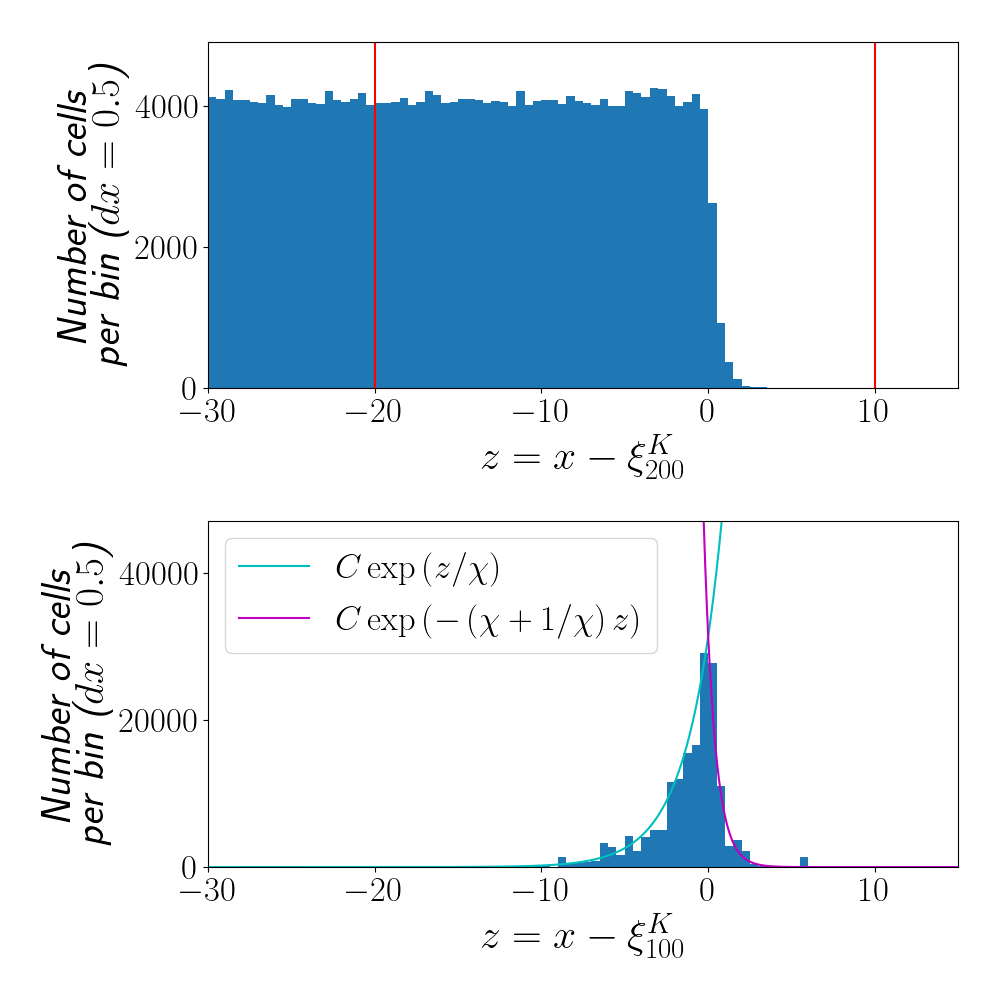}
  \caption{}
  \label{fig:backward-distribution:sfig1}
\end{subfigure}%
\begin{subfigure}{.5\linewidth}
  \centering
\includegraphics[width=\linewidth]{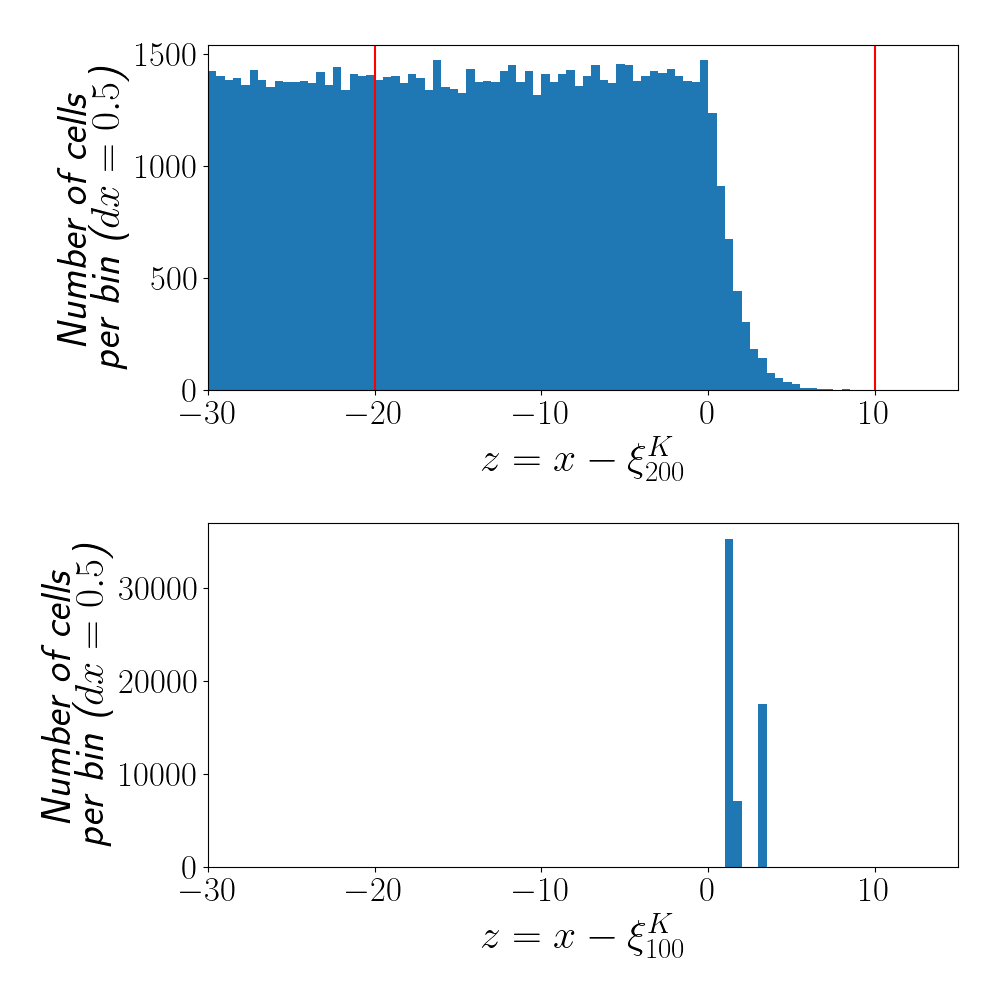}
  \caption{}
  \label{fig:backward-distribution:sfig2}
\end{subfigure}
\caption{Graphical representations of the ancestral distribution. The top figure represents the histogram of $ \mu_T^K$ in the frame $z=x-\xi^K_T$ for $T=200$, $K=4096$ and initial data satisfying $(\text{ID})$. The two red bars stake out the positions of all the particles selected, which are all the particles whose position is in $[-20,10]$.  The bottom figure represents the ancestral distribution in the frame $(t,x-\xi^K_t)$ at time $t=100$, \textit{i.e.} $s=100$, of the selected particles (which are between red bars in the top figure). (a): The case $\chi=2$. The cyan curve represents $y_1(z)=Ce^{\frac{z}{\chi}}$ and the magenta curve represents $y_2(z)=Ce^{-\left(\chi-\frac{1}{\chi}\right)z}$, where $C$ has been chosen \textit{a priori} with $C = n \frac{\chi^2-1}{\chi^3}dx $, so that $v_\infty\left(z\right)=\frac{1}{ndx}\min\{y_1(z),y_2(z)\}$ (see Equation (\ref{def-vinf})), which is consistent with the discretization of the histogram. The distribution of the ancestors seems to follow the distribution $v_\infty$. The number of selected particles is $n=167686$ and the number of distinct ancestors is $2811$.  {(b)}: The case $\chi=\frac{1}{2}$. The ancestors of the particles all come from the leading edge. The number of selected particles is $n=59927$ and the number of distinct ancestors is $7$.}
\label{fig:backward-distribution}
\end{center}
\end{figure}

\subsubsection{The Pushed Case $\chi>1$.}

In the pushed case $\chi>1$, Equation (\ref{ancestPDE}) has an equilibrium given by:

\begin{align}
\label{def-vinf}
v_\infty(z):=\frac{\chi^2-1}{\chi^3}\exp\left(\int_0^z \beta(y)dy\right) =\frac{\chi^2-1}{\chi^3}\exp\left(
\left\{\begin{array}{ll}
\frac{z}{\chi} & \text{ if }z\leq0\\
-\left(\chi-\frac{1}{\chi}\right)z & \text{ if } z>0
\end{array}
\right.
\right)
\end{align}

Following the work \cite{demircigil2022}, it can be shown that the probability distribution of the ancestral distribution $v(t)$ converges exponentially to the equilibrium $v_\infty$. The convergence to the equilibrium is illustrated in Figure \ref{fig:backward-distribution:sfig1}: We consider all particles that at time $t=200$ are located in the interval $[-20,10]$ and plot the histogram of their ancestors position $\hat{Y}^{i,K}_{100,200}$, which closely follows the predicted shape prescribed by $v_\infty$.

\subsubsection{The Pulled Case $\chi\leq 1$.}

In the pulled case $\chi\leq 1$, the stationary solution of Equation (\ref{ancestPDE}) is given by:
\begin{align}
v_\infty(z) := \left\{ 
\begin{array}{ll}
e^{(2-\chi)z} & \text{ if } z<0 \\
\left( (1-\chi)z+1\right)^2 & \text{ if } z\geq 0
\end{array}
\right. .
\end{align}
We immediately observe that $v_\infty\notin L^1(\R)$. Hence, as Equation (\ref{ancestPDE}) is conservative, the probability distribution of the ancestral lineage $v$ cannot converge to the stationary solution $v_\infty$. 
Furthermore, in Figure \ref{fig:backward-distribution:sfig2}, we see the stark difference with the pushed case: Considering again all particles that at time $t=200$ are located in the interval $[-20,10]$ and plotting the histogram of their ancestors position $\hat{Y}^{i,K}_{100,200}$, the distribution does not seem to converge to an equilibrium. Furthermore, we can observe that all the ancestors are to the right of the $K$-th particle, which again illustrates that the pulled wave is driven by the particles at the leading edge of the wave. 

This latter behavior is confirmed by the behavior of the solution $v(s)$ to Equation (\ref{ancestPDE}) (see Figure \ref{fig:v-pulled:sfig1}): $v$ behaves like a bell curve that drifts to the right.

Finally, let us observe that in the \textit{pushmi-pullyu} case $\chi=1$, we observe a qualitative difference to the case $\chi<1$ (see Figure \ref{fig:v-pulled:sfig2}): Whilst $v(s)$ does not converge to an equilibrium and its first moment drifts clearly to the right as well, which is consistent with the pulled regime, the curve of $v(s)$ has a maximum at $z=0$, which it shares with the pushed regime, and therefore, compared to the $\chi<1$ case, many more of its ancestors will remain in the bulk of the traveling wave.

\begin{figure}[t]
\begin{center}
\begin{subfigure}{.5\linewidth}
  \centering
\includegraphics[width=\linewidth]{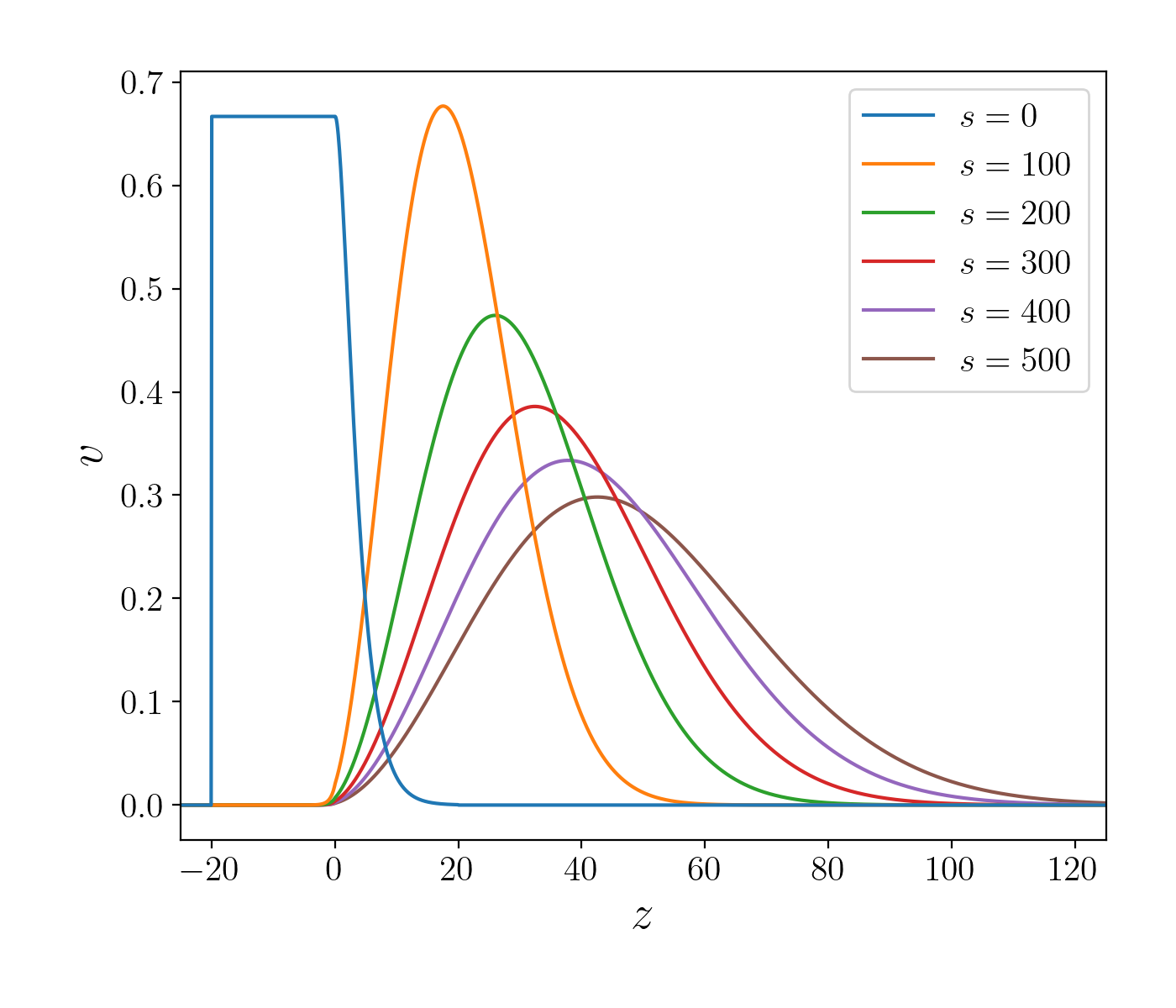}
  \caption{}
  \label{fig:v-pulled:sfig1}
\end{subfigure}%
\begin{subfigure}{.5\linewidth}
  \centering
\includegraphics[width=\linewidth]{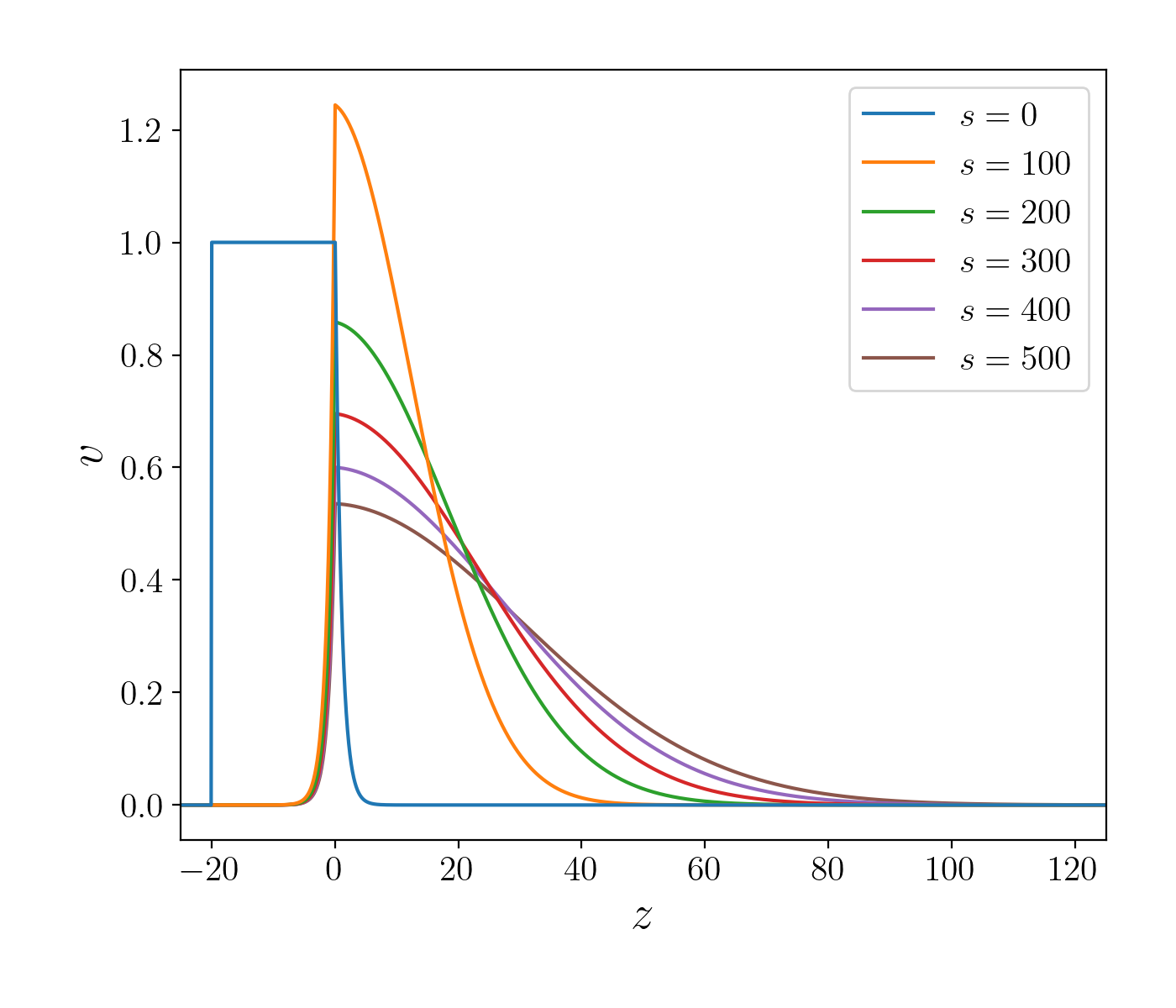}
  \caption{}
  \label{fig:v-pulled:sfig2}
\end{subfigure}
\caption{Solution $v$ of Equation (\ref{ancestPDE}), when $\chi=\frac{1}{2}$ (a) and $\chi=1$ (b), with inital datum $v^0(z)=u^{\sigma^*}(z)\mathbbm{1}_{[-20,10]}(z)$, which corresponds to the distribution of the cells that are staked oud in Figure \ref{fig:backward-distribution}. }
\label{fig:v-pulled}
\end{center}
\end{figure}

%%%%%%%%%%%%%%%%%%%%%%%%%%%%%%%%%%%%%%%%%%%%%%
%% Single Appendix:                         %%
%%%%%%%%%%%%%%%%%%%%%%%%%%%%%%%%%%%%%%%%%%%%%%
%\begin{appendix}
%\section*{???}%% if no title is needed, leave empty \section*{}.
%\end{appendix}
%%%%%%%%%%%%%%%%%%%%%%%%%%%%%%%%%%%%%%%%%%%%%%
%% Multiple Appendixes:                     %%
%%%%%%%%%%%%%%%%%%%%%%%%%%%%%%%%%%%%%%%%%%%%%%
%\begin{appendix}
%\section{???}
%
%\section{???}
%
%\end{appendix}

%%%%%%%%%%%%%%%%%%%%%%%%%%%%%%%%%%%%%%%%%%%%%%
%% Support information, if any,             %%
%% should be provided in the                %%
%% Acknowledgements section.                %%
%%%%%%%%%%%%%%%%%%%%%%%%%%%%%%%%%%%%%%%%%%%%%%
\begin{acks}[Acknowledgments]
The authors deeply appreciate Vincent Calvez for his invaluable guidance, which has been instrumental in shaping the work presented here.
\end{acks}
%%%%%%%%%%%%%%%%%%%%%%%%%%%%%%%%%%%%%%%%%%%%%%
%% Funding information, if any,             %%
%% should be provided in the                %%
%% funding section.                         %%
%%%%%%%%%%%%%%%%%%%%%%%%%%%%%%%%%%%%%%%%%%%%%%
\begin{funding}
This project has received funding from the European Research Council (ERC) under the European Union's Horizon 2020 research and innovation program (grant agreement No 865711).
\end{funding}

%%%%%%%%%%%%%%%%%%%%%%%%%%%%%%%%%%%%%%%%%%%%%%
%% Supplementary Material, including data   %%
%% sets and code, should be provided in     %%
%% {supplement} environment with title      %%
%% and short description. It cannot be      %%
%% available exclusively as external link.  %%
%% All Supplementary Material must be       %%
%% available to the reader on Project       %%
%% Euclid with the published article.       %%
%%%%%%%%%%%%%%%%%%%%%%%%%%%%%%%%%%%%%%%%%%%%%%
%\begin{supplement}
%\stitle{???}
%\sdescription{???.}
%\end{supplement}

%%%%%%%%%%%%%%%%%%%%%%%%%%%%%%%%%%%%%%%%%%%%%%%%%%%%%%%%%%%%%
%%                  The Bibliography                       %%
%%                                                         %%
%%  imsart-???.bst  will be used to                        %%
%%  create a .BBL file for submission.                     %%
%%                                                         %%
%%  Note that the displayed Bibliography will not          %%
%%  necessarily be rendered by Latex exactly as specified  %%
%%  in the online Instructions for Authors.                %%
%%                                                         %%
%%  MR numbers will be added by VTeX.                      %%
%%                                                         %%
%%  Use \cite{...} to cite references in text.             %%
%%                                                         %%
%%%%%%%%%%%%%%%%%%%%%%%%%%%%%%%%%%%%%%%%%%%%%%%%%%%%%%%%%%%%%

%% if your bibliography is in bibtex format, uncomment commands:
\bibliographystyle{imsart-number} % Style BST file (imsart-number.bst or imsart-nameyear.bst)
\bibliography{biblio}       % Bibliography file (usually '*.bib')

%% or include bibliography directly:
% \begin{thebibliography}{}
% \bibitem{b1}
% \end{thebibliography}

\end{document}